\documentclass[12pt,a4paper,oneside]{amsart}

\usepackage{amsfonts, amsmath, amssymb, amsthm, amscd}
\usepackage{hyperref}
\hypersetup{
  colorlinks   = true, %Colours links instead of ugly boxes
  urlcolor     = blue, %Colour for external hyperlinks
  linkcolor    = blue, %Colour of internal links
  citecolor   = red %Colour of citations
}
\usepackage{anysize}
\marginsize{1.8cm}{1.8cm}{1.8cm}{1.8cm}
\usepackage[pdftex]{graphicx}

\newtheorem{theorem}{Theorem}[section]
\newtheorem{lemma}[theorem]{Lemma}

\newtheorem{corollary}[theorem]{Corollary}

\theoremstyle{definition}

\theoremstyle{remark}
\newtheorem{remark}[theorem]{Remark}

\numberwithin{equation}{section}

\DeclareMathOperator{\vol}{vol}

\renewcommand{\epsilon}{\varepsilon}

\renewcommand{\phi}{\varphi}

\newcommand{\lmink}{\mathop{\underline{\mathcal M}}\nolimits}
\newcommand{\haus}{\mathop{\mathcal H}\nolimits}
\newcommand{\sys}{\mathop{\mathrm{Sys}}\nolimits}
\newcommand{\skel}{\mathop{\mathrm{skel}}\nolimits}
\newcommand{\pr}{\mathop{\mathrm{pr}}\nolimits}

\newcommand{\R}{\mathbb{R}}
\newcommand{\ZZ}{\mathbb{Z}}

\newcommand{\Sphere}{\mathbb{S}}

\title{Lower and upper bounds for the waists of different spaces}

\author{Arseniy Akopyan{$^\spadesuit$}}

\author{Alfredo Hubard{$^\clubsuit$}}

\author{Roman Karasev{$^\diamondsuit$}}

\thanks{{$^\spadesuit$} Supported by People Programme (Marie Curie Actions) of the European Union's Seventh Framework Programme (FP7/2007--2013) under REA grant agreement n$^\circ$[291734].}
\thanks{{$^\clubsuit$} Supported by the Advanced Grant of the European Research Council GUDHI (Geometric Understanding in Higher Dimensions).}
\thanks{{$^\diamondsuit$}Supported by the Russian Foundation for Basic Research grant 15-01-99563 A}

\address{{$^\spadesuit$} Institute of Science and Technology Austria (IST Austria), Am Campus 1, 3400 Klosterneuburg, Austria}
\email{akopjan@gmail.com}

\address{{$^\clubsuit$}Universit\'e Paris-Est Marne-la-Vall\'ee}
\email{alfredo.hubard@u-pem.fr}

\address{{$^\diamondsuit$}Moscow Institute of Physics and Technology, Institutskiy per. 9, Dolgoprudny, Russia 141700
\newline \indent
    Institute for Information Transmission Problems RAS, Bolshoy Karetny per. 19, Moscow, Russia 127994}
\email{r\_n\_karasev@mail.ru}
\urladdr{http://www.rkarasev.ru/en/}


\begin{document}

\begin{abstract}
We prove several new results around Gromov's waist theorem. We give a simple proof of Vaaler's theorem on sections of the unit cube using the Borsuk--Ulam--Crofton technique. We consider waists of real and complex projective spaces, flat tori, convex bodies in Euclidean space. We establish waist-type results in terms of the Hausdorff measure.
\end{abstract}

\maketitle

\section{Introduction}

Using his version of Morse theory, Almgren showed that for any smooth map $f:\mathbb S^n \to \mathbb R^k$, under certain general position assumptions, there exists a $y \in \mathbb R^k$ such that the Riemannian volume of the fiber $f^{-1}(y)$ is larger than the Riemannian volume of the sphere $\mathbb S^{n-k}$ (here $\mathbb S^i:=\{x \in \mathbb R^{i+1}: |x|=1\}$), and this bound on 
\[
\inf_{f} \sup_{y \in  Y}  \vol_{n-k} (f^{-1}(y)),
\] 
is achieved by a linear projection restricted to $\mathbb S^n$.

Almgren's work was motivated by the quest of minimal submanifolds in Riemannian manifolds. Gromov focused on the inherent interest of the aforementioned result, he observed the similarity with other geometric invariants like the Urysohn width, the Alexandrov width, the Cheeger constant, etc. he coined the term \emph{waist} (sometimes called the width as both words translate from the Russian ``poperechnik''),  and wondered if the same result holds true for all continuous maps.  In the case of continuous maps fibers might not be rectifiable, so a different way to measure volume is needed. Gromov first considered the waist measured by the Hausdorff measure.  He used his filling technique to show the existence of a constant $\epsilon_{n,k}$ so that for every continuous map $f$ there exists a fiber $f^{-1}(y)$ with $(n-k)$-dimensional Hausdorff measure, denoted by $\haus^{n-k}(f^{-1}(y))$ is at least $\epsilon_{n,k} \haus^{n-k}(\mathbb S^{n-k})$. It is an open problem if this statement remains true for $\epsilon_{n,k}=1$.  Years later ~\cite{grom2003}, he considered the waist measured by the Minkowski content, and introduced the $t$-neighborhood waist (the $t$-waist for briefness): Given a measure metric space $X$, a topological space $Y$, and a set of maps $\Gamma$ from $X$ to $Y$, the \emph{$t$-neighborhood waist} of $\Gamma$ can be generally defined by 
\[
\inf_{f \in \Gamma} \sup_{y \in  Y}  \mu(\nu_t f^{-1}(y)),
\]
where $\nu_t f^{-1}(y)$ is the set of points at distance less than $t$ from $f^{-1}(y)$. Estimating $t$-neighborhood waists of a space of maps is a challenging task even for simple spaces of maps. In~\cite{grom2003} Gromov showed that the $t$-neighborhood waist of continuous maps $C(\mathbb S^n,\mathbb R^k)$ is larger than $\vol(\mathbb \nu_tS^{n-k})$, where $\vol$ is the Riemannian volume in $\mathbb S^n$. From this theorem, Almgren's theorem (for fibers of a map to a manifold) can be recovered by taking $t\to 0$, as it implies that the \emph{lower Minkowski content} (not a measure, strictly speaking) of a fiber is at least that of an equatorial $\mathbb S^{n-k}$ sphere, where the lower Minkowski content is defined as
\[
\lmink_{n-k} X = \liminf_{t \to+0} \frac{\vol_n(\nu_t X)}{v_{k} t^k},
\]
where $v_k=\frac{\pi^{k/2}}{(k/2)!}$ is the volume of the unit Euclidean $k$-dimensional ball. 

\subsection{This paper}

In this note we consider analogous waist inequalities for a number of other spaces. Firstly we observe that under symmetry hypothesis on the map $f$, the waist inequality for Hausdorff measure with $\epsilon_{n,k}=1$ can be easily derived from topological considerations and Crofton's formula. A variation of this technique provides as a corollary a short proof of a famous theorem of Vaaler on the linear waist of the cube (every central linear section of the unit cube has at least unit volume), as well as a nonlinear generalization. 

This is followed by our study of the waists of a number of other spaces\footnote{For much more information on waist inequalities see Gromov's papers, and Guth's essay~\cite{guth2014} and papers.}: real and complex projective spaces, rectangular flat tori, and convex bodies. Many of our results concern the waist measured by the Minkowski content developing the ideas of~\cite{klartag2016,ak2016ball}, where the waist of rectangular parallelotopes and ellipsoids (in the sense of the lower Minkowski content) was determined.
 
In a subsequent section we study a different generalization: waist inequalities in terms of topological invariants of the space of cycles of a space. We focus on the space of cycles of spheres, balls, and cubes. Waist inequalities in terms of the space of cycles are more general then the waist inequalites for smooth (and sufficiently regular) maps. Then we make some observations about the known results on the volume of other cohomology classes of the space of cycles. The point we want to make here is that the same Minkowski content considerations that are previosly used in the paper carry to this context. 

Finally, we come back to the Hausdorff measure. We generalize a result of Lebesgue in the following way:  we prove that there exist $\varepsilon_{n,\delta}>0$, such that for every finite covering of the cube $[0,1]^n$ by family of closed sets $\{C_i\}$ so that each set $C_i$ intersects at most $\delta$ of the other sets,  there exists $k\in\{0,\ldots, n\}$, and indices $i_0<i_1<\dots<i_k$ such that $\haus^{n-k}(C_{i_0}\cap \dots \cap C_{i_k})> \varepsilon_{n,\delta}$, where as before $\haus^{n-k}$ denotes the $(n-k)$-dimensional Hausdorff measure.

\subsection{Contribution and organization}
With hindsight we suspect that some of the results of this paper are folklore. This is probably the case for theorems ~\ref{theorem:bu-hausdorff} \ref{theorem:rp3functions} \ref{corollary:rpn-gromov}, ~\ref{theorem:rpn-hopfmaps}. Many of our results seem new and come as natural corollaries and extensions of the technique used in ~\cite{klartag2016,ak2016ball}, this is the case for ~\ref{theorem:non-linear-vaaler-cube},~\ref{theorem:non-linear-vaaler-balls},  ~\ref{corollary:cpn-gromov}, ~\ref{corollary:tortho}, ~\ref{theorem:tor2}, and the results of Section 6. A third type of results are new and combine LS-category type results with geometric observations, this is the case for \ref{theorem:rp2functions} and \ref{theorem:1waist}. Finally, Section  \ref{section:hausdorff-cube} involves a generalization of the Lebesgue lemma using Gromov's filling technique; this is also new although it mimics the analogous result for the discrete cube from~\cite{kar2013cube}.

%In the case $k=1$, it was observed by Gromov that the Hausdorff measure estimate is established for $f : \mathbb S^n \to \mathbb R$ by taking $y$ to be the median value of $f$ and applying the spherical isoperimetric inequality for the Hausdorff measure. In the case $k=n-1$, Alexandrov's width (waist) theorem (see \cite[Section 6]{kar2012}) guarantees that a connected component of some fiber $f^{-1}(y)$ cannot be covered by a smaller homothet of the cube, hence its $1$-Hausdorff measure is at least $\varepsilon_{n-1,n} = 1$; this estimate is evidently tight. In general \emph{we have no evidence} that the constant $\varepsilon_{k,n}$ has to be less than $1$ in the case of general target spaces $Y$.

%We prove that there exist $\varepsilon_{n,\delta}>0$, such that if the cube $Q = [0,1]^n$ is covered by a finite family of closed sets $\{C_i\}$ so that the degree of the intersection graph\footnote{The \emph{intersection graph} of a covering is the $1$-skeleton of its nerve.} is bounded by $\delta$. Then there exists $k\in\{0,\ldots, n\}$, and indices $i_0<i_1<\dots<i_k$ such that the $(n-k)$-Hausdorff measure of the intersection 
%$$
%C_{i_0}\cap \dots \cap C_{i_k}
%$$
%is at least $\varepsilon_{n,\delta}$.
 
In Section~\ref{section:hausdorff-odd} we describe our observation for the Hausdorff waist of odd maps and our simplified proof of Valeer's theorem; Sections \ref{section:projective}--\ref{section:tori} contain our results on the  waist  in terms of Minkowski content of certain homogeneous spaces, mostly real and complex projective spaces and orthogonal flat tori. In Section~\ref{section:convex} we give some new results on the waists of convex bodies. In Section~\ref{section:sweepouts} we consider the waist in terms of family of cycles that sweepout a manifold (possibly with boundary). We show that the techniques of~\cite{klartag2016,ak2016ball} extend to this case, and improve the estimates for the higher waists of the cube. Section~\ref{section:hausdorff-cube} presents our results of Hausdorff-measure waists of maps from the cube to a polyhedra.

\subsection{Acknowledgements} The authors thank Bo'az Klartag for the numerous discussions and helpful remarks, including the proof of Theorem~\ref{theorem:1waist-cs}. Alfredo Hubard thanks Yashar Memarian for interesting conversations.

\section{Real projective space (and odd maps on the sphere)}
\label{section:projective} 

\subsection{Odd maps}\label{section:hausdorff-odd}

Let us start with a relatively simple result. It seems essentially known, see for example \cite{barlov1982}, where a very similar argument is used to give a lower bound on the number of vertices of a linear section of a cube, and \cite[Section 5]{guth2010}, where a similar claim is made in terms of subspaces of the real projective space.

\begin{theorem} 
\label{theorem:bu-hausdorff}
For any odd continuous map $f\colon \Sphere^n \to \R^k$, 
\[
\haus^{n-k}(f^{-1}(0))\geq \haus^{n-k}(\Sphere^{n-k}).
\]
\end{theorem}

Thus, for odd maps, there is a tight Hausdorff measure estimate for the waist of the sphere.
 
\begin{proof}
Let $E\subset \mathbb S^n$ be any equatorial subsphere of dimension $k$. By the Borsuk--Ulam theorem the restriction $f|_E : E\to \mathbb R^k$ has a zero, that is $f^{-1}(0)\cap E\neq \emptyset$. In fact, $f^{-1}(0)\cap E$ contains at least two opposite points. Now recall Crofton's formula\footnote{Crofton's formula provides $\mathcal I^{n-k}(Z)=\vol_{n-k}(Z)$ when $Z$ is rectifiable. In our context the formula provides the definition of the integral-geometric volume.}, 
\[
\mathcal I^{n-k}(f^{-1}(0)):= \frac{1}{2}{\vol_{n-k}(\Sphere^{n-k})} \int_{E\in G(k+1,n+1)} \#(E \cap f^{-1}(0)) d\mu(E)\geq \vol_{n-k}(\Sphere^{n-k}),
\] 
here $E$ is an equatorial subsphere of dimension $k$, $\mu$ is the probability measure on such subspheres of $\Sphere^n$ that is invariant under the action of the group $SO(n+1)$ on $\Sphere^n$ (the Haar measure), and $\#(E \cap f^{-1}(0))$ denotes the number of intersection points between $E$ and $f^{-1}(0)$. 

Now we recall that Federer's structure theorem~\cite[Theorems 3.17 and 3.16]{morgan2009} (referring to~\cite[Theorem 3.2.26 and 3.3.13]{federer1969}) implies that either the Hausdorff measure $\haus^{n-k} (f^{-1}(0))$ is infinite or 
$$
\haus^{n-k} (f^{-1}(0)) \ge \mathcal I^{n-k} (f^{-1}(0)) \ge  I^{n-k} (\mathbb S^{n-k})= \haus^{n-k} (\mathbb S^{n-k}).
$$
\end{proof}

\begin{remark}
In the above argument it is possible to take any $k$-dimensional manifold in place of $\mathbb R^k$, for $k<n$, in the standard manner. We refer to~\cite{karvol2013} for the details. The case of an arbitrary $k$-dimensional polyhedron on the right hand side is not clear.
\end{remark}

\begin{remark}
The same argument substituting the integral geometry volume by the estimate on the volume coming from the systolic inequality gives the following more general statement (except for the sharp quantitative dependence): For any  $\ZZ_2$-invariant metric $m$ on the sphere $\Sphere^n$ and any piecewise smooth $\ZZ_2$-continuous map $f\colon \Sphere^n \to \mathbb R^k$, we have 
$$
\vol_{n-k}(f^{-1}(0))\geq \epsilon'_{n,k} \sys(m)^{n-k}.
$$
It is conjectured that the extremal metric for the systolic inequality on the projective space is given by the standard round metric.   %The corresponding waist inequalities for Finsler metrics on the sphere with different notions of volume might be of interest in convex geometry, the case $Y=\mathbb R$ is related to the KLS conjecture. This connection is explored in detail in Klartag's recent paper.\\
In the rest of the paper we do not make allusion to non-canonical metrics. We always refer to the standard round metric on the sphere or its $\ZZ_2$ quotient and to the Euclidean metric on the cube.
\end{remark}

\subsection{The Borsuk--Ulam--Crofton approach to Vaaler's theorem}

Let us make an observation about the usefulness of Theorem~\ref{theorem:bu-hausdorff}; here we apply it to the maps that have smooth $f^{-1}(0)$ and may therefore speak about the Riemannian volume instead of the Hausdorff measure. The idea is to use the Borsuk--Ulam--Crofton argument to establish a version of the waist theorem for a radially symmetric measure in $\mathbb R^n$ (we particularly need the case of the Gaussian measure); and then use the (folklore) trick of transporting the Gaussian measure to the unit cube to infer Vaaler's theorem.

Assume we have a density $\rho : \mathbb R^n\to\mathbb R^+$, which only depends on the distance from the origin. We will actually discuss the $\rho$-weighted version of the Riemannian volume, given by integrating this density by the Riemannian volume measure.

\begin{lemma}
\label{lemma:odd-radial}
Let $\rho$ be a radially symmetric density in $\mathbb R^n$. For any odd continuous map $f : \mathbb R^n \to \mathbb R^k$, if the set $Z = f^{-1}(0)$ is a (piece-wise) smooth $(n-k)$-dimensional manifold then we estimate its $\rho$-weighted Riemannian $(n-k)$-volume as 
\[
\vol_{n-k}^\rho Z \ge \vol_{n-k}^\rho \mathbb R^{n-k},
\]
where we mean a standard linear subspace $\mathbb R^k\subset\mathbb R^n$.
\end{lemma}
\begin{proof}
For any sphere $S_r^{n-1}$ of radius $r$ centered at the origin, the Riemannian $(n-k-1)$-volume of $Z\cap S_r^n$ is at least $\vol_{n-k-1} \mathbb S^{n-k-1} r^{n-k-1}$ by Theorem~\ref{theorem:bu-hausdorff}. Integrating this we easily obtain:
\[
\vol_{n-k}^\rho Z := \int_Z \rho(x) \vol_{n-k} \ge \int_0^{+\infty} \vol_{n-k-1} \mathbb S^{n-k-1} \rho(r) r^{n-k-1}\; dr = \vol_{n-k}^\rho \mathbb R^k.
\]
\end{proof}

For example, taking the density $\rho(x) = e^{-\pi |x|^2}$, we have the estimate $\vol_{n-k}^\rho Z \ge 1$ for the (piece-wise smooth) zero set of any odd map $\mathbb R^n\to\mathbb R^k$, compare this with a much harder result for zeros of holomorphic maps in~\cite{klartag2017}. Now we show that this simple approach also works in some cases without radial symmetry.

\begin{theorem}
\label{theorem:non-linear-vaaler-cube}
Let $g : (-1/2, 1/2)^n \to \mathbb R^k$ be an odd continuous map, such that $Z = g^{-1}(0)$ is piece-wise smooth and $(n-k)$-dimensional. Then $\vol_{n-k} Z \ge 1$.
\end{theorem}

\begin{proof}
Let us transform the Euclidean space with the Gaussian density to the cube with the uniform density, using the folklore trick that is mentioned in~\cite[Remark before Section 4]{barthemaurey2000}, \cite[Theorem 7]{ros2001}, and applied to the waist in~\cite{klartag2016}. Put 
\[
y(x) = \int_{0}^x e^{-\pi \xi^2}\; d\xi.
\]
This function maps $\mathbb R$ to $(-1/2, 1/2)$, has derivative at most $1$ everywhere, and the push-forward of the measure with density $e^{-\pi x^2}$ is the uniform measure on $(-1/2, 1/2)$. The map
\[
T : \mathbb R^n \to (-1/2, 1/2)^n,\quad T(x_1,\ldots,x_n) = (y(x_1), \ldots, y(x_n))
\]
is $1$-Lipschitz (does not increase the distances between pairs of points) and the push-forward of the Gaussian measure under $T$ is evidently the uniform measure on the open cube $(-1/2, 1/2)^n$.

Consider the composition $f = g\circ T : \mathbb R^n \to \mathbb R^k$ and the set $Z' = f^{-1}(0)$, then $T(Z') = Z$. Since $f$ is an odd map, we have by the above argument for $Z'$ that $\vol_{n-k}^\rho Z'\ge 1$. It remains to examine how $T$ transforms the $(n-k)$-dimensional Riemannian volume.

The fact that $T$ transforms the Gaussian measure to the uniform measure in the cube is expressed with $\det DT = \rho$. In order to check the transform of the $(n-k)$-dimensional Riemannian volume we consider an $(n-k)$-dimensional subspace $L$ in the tangent space. 

Since all eigenvalues of $DT$ are no greater than $1$ (this is a Lipschitz map), its determinant cannot decrease when going to the subspace, that is $\det DT|_L \ge \rho$.\footnote{Here we consider the determinant of a map $L : V\to W$ between Euclidean spaces of different dimension as $\sqrt{\det L^* L}$ assuming the identifications $W=W^*$ and $V=V^*$ under the Euclidean structure.} Indeed, when considering the quadratic form given by $DT^* DT$ on a linear space $V$, we may write down the eigenvalues $0 \le \lambda_1 \le \dots \le \lambda_\ell\le 1$. Passing to a subspace of one dimension smaller, we consider the new eigenvalues $\mu_1,\ldots,\mu_{\ell-1}$. Using the formula for the ordered eigenvalues of a quadratic form $Q:V\to \mathbb R$ on e Euclidean space $V$
\[
\lambda_k(Q) = \min_{W\subset V,\ \dim W = k} \left( \max_{w\in W\setminus 0} \frac{Q(w)}{|w|^2} \right)
\]
we readily obtain the inequalities connecting the original eigenvalues and the eigenvalues of a restriction:
\[
0\le \lambda_1 \le \mu_1 \le \lambda_2 \le \mu_2 \le \dots \le \mu_{\ell-1}\le \lambda_\ell \le 1.
\]
From this it follows that 
\[
\mu_1\dots \mu_{\ell-1} \ge \lambda_1\dots\lambda_{\ell-1} \ge \lambda_1\dots\lambda_{\ell-1}\lambda_\ell
\]
and therefore $\det DT^* DT$ cannot decrease when restricting to a subspace.

The inequality $\det DT|_L \ge \rho$ means that the resulting $(n-k)$-volume is no smaller that the original $\rho$-weighted $(n-k)$-volume, showing that $\vol_{n-k} Z \ge 1$.
\end{proof}

A particular case of the above theorem is Vaaler's theorem~\cite{vaaler1979} about sections of a cube: 

\begin{corollary}[Vaaler, 1979]
Every $(n-k)$-dimensional linear section of the unit cube $(-1/2,1/2)^n$ has Riemannian $(n-k)$-volume at least $1$. 
\end{corollary}

Compare the above argument to~\cite{klartag2016}, where Vaaler's theorem is inferred from Gromov's waist theorem for the Gaussian measure. The approach presented here avoids the use of hard Gromov's waist theorem, using the Borsuk--Ulam theorem and Crofton's formula instead.

In fact, the full version of Vaaler's theorem gives an estimate for a linear section of a Cartesian product of Euclidean balls of arbitrary dimensions, but all of unit volume. The Borsuk--Ulam--Crofton approach produces the corresponding generalization:

\begin{theorem}
\label{theorem:non-linear-vaaler-balls}
Let $g : K \to \mathbb R^k$ be an odd continuous map, such that $Z = g^{-1}(0)$ is piece-wise smooth and $(n-k)$-dimensional, where $K\subset \mathbb R^n$ is a Cartesian product $K=B^{n_1}\times\dots\times B^{n_k}$ of open Euclidean balls, each $B^{n_i}$ having $n_i$-volume $1$. Then $\vol_{n-k} Z \ge 1$.
\end{theorem}

\begin{proof}
The scheme of the proof is the same as in Theorem~\ref{theorem:non-linear-vaaler-cube}; we only need to transport the Gaussian measure with density $e^{-\pi |x|^2}$ on $\mathbb R^{n_i}$ to the uniform measure in a unit volume ball $B^{n_i}$ with a map $T_i$. This transportation has to be radially symmetric and the corresponding map of the radial coordinate $x\mapsto y$ is determined by the equation
\[
\int_0^y nr^{n-1} dr = \int_0^x nr^{n-1} e^{-\pi r^2} dr,
\]
or a bit more explicitly
\[
y^n = \int_0^x nr^{n-1} e^{-\pi r^2} dr.
\]
It is important for this map to be $1$-Lipschitz. The Lipschitz property in directions orthogonal to the radial direction is guaranteed by the trivial observation $y\le x$. The radial direction requires differentiation of the defining relation to give
\[
y'_x = \frac{nx^{n-1} e^{-\pi x^2}}{ny^{n-1}} = \frac{yx^{n-1} e^{-\pi x^2}}{\int_0^x e^{-\pi r^2} d r^n} \le \frac{x^n e^{-\pi x^2}}{\int_0^x e^{-\pi r^2} d r^n} \le \frac{x^n e^{-\pi x^2}}{\int_0^x e^{-\pi x^2} d r^n} = 1,
\]
where we used $y\le x$ and $e^{-\pi r^2}\ge e^{-\pi x^2}$ (for $r\in [0, x]$) in the estimates.

After that we consider $T = T_1\times\dots \times T_k$, which $1$-Lipschitz transports the Gaussian measure to the uniform measure in $K$, apply the Borsuk--Ulam--Crofton argument to $g\circ T$ precisely as in Theorem~\ref{theorem:non-linear-vaaler-cube}.
\end{proof}

Note that the estimate of Theorem \ref{theorem:non-linear-vaaler-balls} for the product of balls is not tight, in particular, it is not tight in the case of a single ball. Of course, for odd maps from the single unit ball Lemma~\ref{lemma:odd-radial} gives a tight estimate, and for arbitrary maps the estimate (for the Minkowski content of a fiber) is produced in~\cite{ak2016ball} by considering the natural (Archimedes) projection $\mathbb S^{n+1}\to B^n$ and using the fact that it pushes forward the uniform measure on the sphere to the uniform measure on the ball in a $1$-Lipschitz way.

\subsection{Even maps}

To make one further step we consider even maps $g : \mathbb S^n \to Y$, where $Y$ is a $k$-dimensional polyhedron. Of course, this map factors through a map $f : \mathbb RP^n \to Y$. So the question for even maps is equivalent to the question of the waist inequality for the projective space. 

In~\cite{kar2012} it was shown, in particular, that for 
$$
m = \left\lfloor \frac{n}{k+1}\right\rfloor,
$$
there exists $y\in Y$ such that $t^m$ does not vanish in the \v{C}ech cohomology of $f^{-1}(y)$, where $t\in H^1(\mathbb RP^n; \mathbb F_2)$ is the cohomology generator. This in the standard way enables the Borsuk--Ulam property for odd continuous maps $g^{-1}(y) \to \mathbb R^m$, but the problem is that $m$ is usually much smaller than the expected dimension of $f^{-1}(y)$, $n-k$. The equality $m=n-k$ only holds when $k=n-1$. In this case we have: 

\begin{theorem}
Let $g \colon \mathbb RP^n\to Y$ be a continuous map to a polyhedron dropping the dimension by one. Then for some $y\in Y$ the length of the fiber $g^{-1}(y)$ is at least $\pi$.
\end{theorem}

\begin{proof}
From~\cite{kar2012} we find a fiber $g^{-1}(y)$ such that the restriction of the generator of $H^1(\mathbb RP^n; \mathbb F_2)$ to $g^{-1}(y)$ is nonzero. Therefore its lift to $S^n$, $f^{-1}(y)$, intersects every rotated image of $S^{n-1}$. The Crofton formula then implies that the $1$-Hausdorff measure of $f^{-1}(y)$ is at least $2\pi$ and the $1$-Hausdorff measure of $g^{-1}(y)$ is at least $\pi$.
\end{proof}

\begin{remark}
This estimate is attained for the Hopf map $S^{2n+1}\to \mathbb CP^n$ and its quotient $\mathbb RP^{2n+1}\to \mathbb CP^n$. What about even-dimensional spheres?
\end{remark}

It seems that the same estimate on the $1$-dimensional Hausdorff measure for $n=k+1$ should hold for not necessarily even continuous maps $f : \mathbb S^n \to Y$ with $\dim Y = n-1$. Following~\cite[Theorem~7.3]{akv2012}, we can use Alexandrov's width theorem to find $y\in Y$ such that $f^{-1}(y)$ is not contained in an open hemisphere of $\mathbb S^n$, this is sufficient to conclude that 
$$
\haus^1(f^{-1}(y)) \ge \mathcal I^1(f^{-1}(y)) \ge \pi,
$$
but this estimate is twice smaller than what we expect. Notice that while the previous result for odd maps is tight\footnote{The restriction of a linear projection $f\colon \mathbb S^n \to \mathbb R^k$ is an odd map.}, the inequality 
$\haus^1(f^{-1}(y)) \ge 2\pi$ seems to be not tight. Constructing a map with small waist is not immediate.  In fact, we can prove something particular under certain regularity assumptions:

\begin{theorem}
\label{theorem:rp2functions}
Let $\mathbb RP^2\to\mathbb R$ be a real-analytic map and let $\epsilon>0$. Then for some $y\in\mathbb R$ the length of the fiber $f^{-1}(y)$ is at least $2\pi - \epsilon$.
\end{theorem}

\begin{proof}[Proof of Theorem~\ref{theorem:rp2functions}]
As in the previous argument (referring to~\cite{kar2012}) there exists $y$ such that $f^{-1}(y)$ carries a nontrivial homology of $H_1(\mathbb RP^2;\mathbb F_2)$. Let this $y$ be equal to $0$ without loss of generality. Under the assumption that $f$ is real-analytic, the set $f^{-1}(0)$ is 

\begin{itemize}
\item
either $2$-dimensional; then its length ($1$-Hausdorff measure) is infinite and we are done;
\item
or a $1$-dimensional graph, thus we restrict the argument to this case.
\end{itemize}

This graph has a simple cycle $\gamma$ representing a nontrivial homology of $\mathbb RP^2$. Lifting the picture to $\mathbb S^2$ we see that the lift of $\gamma$, $\tilde \gamma$, breaks it into two equal parts $N$ and $S$, interchanged by the map $x\mapsto -x$. From here on we consider $N$ with its boundary $\tilde \gamma$ instead of $\mathbb RP^2$ and the graph $G$ that $f^{-1}(0)$ induces on $N$ and its boundary.

The graph $G$ breaks $N$ into open connected components, on every one of them the function $f$ is either positive or negative. If we consider $f^{-1}(y)$ for a small positive $y$ then its length will tend to the sum of perimeters of the positive components. If we consider $f^{-1}(y)$ for a small by absolute value negative $y$ then its length will tend to the sum of perimeters of the negative components. Hence we are going to prove that either of the sums of perimeters is $\ge 2\pi$.

Assuming the contrary, that both sums of the perimeters are $<2\pi$, let us make some reductions. If an edge of $G$ separates two components with the same sign, positive or negative, just remove it, making the sum of the perimeters less. If a vertex of $G$ has degree greater than $3$ then perturb it without changing the components so that it splits into a set of vertices of degree at most $3$. Such a perturbation can be made so that to keep both the perimeter sums $<2\pi$.

After the reduction process no vertex will have degree greater than $3$. In fact, no vertex in the interior of $N$ will have degree $3$, just because in this case some of the edges from this vertex separates the components of the same sign and can be removed. Eventually the topological picture becomes very special: The edges of $G$ are either parts of $\tilde\gamma$, or edges in the interior of $N$ from one point of $\tilde\gamma$ to another point of $\tilde\gamma$, or loops inside $N$. If there is a loop we may just drop it with everything in its interior, so we assume we only have $\tilde\gamma$ and several pairwise non-intersecting chords of $\tilde\gamma$.

The curve $\tilde\gamma$ has the involution $x\mapsto -x$ induced from the sphere $\mathbb S^2$ and $\tilde\gamma$ with this involution can be viewed (more precisely: equivariantly homeomorphic) as the ordinary $1$-sphere $\mathbb S^1\subset\mathbb R^2$ with the involution $x\mapsto -x$. Let us call the points $x$ and $-x$ interchanged by this involution \emph{antipodal}. Call a segment of $\tilde\gamma$ \emph{short} if it contains no pair of opposite points.

Now we have two alternatives (resembling the Alexandrov width theorem):

\begin{itemize}
\item
Some closure of a component of the complement to $G$, let it be $P$, contains two antipodal points $a,b\in\tilde\gamma$. Consider this situation on the sphere $\mathbb S^2$ and note that the boundary $\partial P$ contains two antipodal points and therefore its two parts, passing from $a$ to $b$, and from $b$ to $a$ both have length $\ge \pi$. So the total perimeter of $P_i$ is at least $2\pi$ and we are done.
\item
Some closure of the component, let it be $P$ again, intersects $\tilde\gamma$ by a collection of segments $I_1,\ldots, I_m$ so that the complement $\tilde\gamma\setminus P$ consists of short segments $J_1,\ldots, J_m$. Assuming we are not in the first case, $I_i\cap -I_j = \emptyset$ for every $i$ and $j$. In this case we return to the whole sphere $\mathbb S^2$ and note that $f^{-1}(0)$ contains the graph $G' = \tilde\gamma\cup \partial P \cup -\partial P$ and $G'$ has the property that every equatorial subpshere $\mathbb S^1\subset \mathbb S^2$ intersects it at least four times. Indeed, if $\mathbb S^1$ intersects $\tilde\gamma$ four times then we are done. Otherwise it intersects $\tilde\gamma$ twice (because $\tilde\gamma$ it is centrally symmetric and there must be an intersection) and must also intersect some pair $J_i, -J_i$; to the latter pair their correspond two edges of $G'\setminus \tilde\gamma$ that are also intersected by $\mathbb S^1$. Applying the Crofton formula, we conclude that $f^{-1}(0)$ itself has length at least $2\pi$.
\end{itemize}
\end{proof}

\begin{remark}
The map $S^2\to \mathbb R$ given by the square of the coordinate $g(x_1,x_2,x_3) = x_1^2$ induces a map $f:\mathbb RP^2\to \mathbb R$ showing that the estimate in the previous theorem cannot be improved. Note that the standard method of establishing waist estimates for real valued functions by taking the median $m$ of $f$ and applying the isoperimetric inequality to the set $\{f(x)\le m\}$ does not give the right bound in this case. It turns out that for $n=3$ the isoperimetric inequality does provide the correct estimate. As far as we know this is the only essential case where the isoperimetric inequality of projective space is understood:
\end{remark}

%As in the case of the sphere, the choice of the median value and the isoperimetric inequality on $\mathbb RP^n$ for sets of half volume provides a lower bound for the waist of even maps from $\mathbb S^n$ to $\mathbb R$, or just sufficiently regular maps from $\mathbb RP^n$ to $\mathbb R$. As we have seen this lower bound does not match the upper bound for $n=2$.

\begin{theorem}
\label{theorem:rp3functions}
Let $f \colon \mathbb RP^3\to\mathbb R$ be a real-analytic map. Then for some $y\in\mathbb R$ the area of the fiber $f^{-1}(y)$ is at least $\pi^2$.
\end{theorem}

\begin{proof}
This time the median argument works. We choose the median value $m$ such that the volume of 
\[
N = \{x : f(x) \le m\}
\]
is half of the volume of $\mathbb RP^3$. An isoperimeteric inequality will give us a certain estimate for the area of $\partial N\subseteq f^{-1}(m)$; and Ritor\'e and Ros~\cite{rosrit1992} provide us such: A subset of $\mathbb R P^3$ with volume half of the full volume has boundary with area greater of equal to the area of the quotient of the torus. 
%Passing to the double cover we see $\mathbb S^3$ as the unit ball of the complex space $\mathbb C^2$ and $\mathbb S^1\left(\frac{1}{\sqrt{2}}\right) \times \mathbb S^1\left(\frac{1}{\sqrt{2}}\right) \subset \mathbb C^2$ turns out to be the unique minimizer with area $2\pi^2$.
\end{proof}

\begin{remark}
Observe that the function $f(z_1,z_2)=|z_1|$, satisfies $f^{-1}(t)= \mathbb S^1(t) \times \mathbb S^1(\sqrt{1-t^2})$ for $t \in [0,1]$, an elementary calculation yields that the bound $\pi^2$ in Theorem~\ref{theorem:rp3functions} is best possible, i.e. for maps with enough regularity, 
\[
\inf_{f \colon \mathbb RP^3 \to \mathbb R} \sup_y \vol_2(f^{-1}(y))=\pi^2.
\]
\end{remark}

A trivial corollary of Gromov's waist of the sphere theorem~\cite{grom2003,mem2009} (in view of~\cite{karvol2013}):

\begin{corollary}
\label{corollary:rpn-gromov}
For any continuous map $f:\mathbb RP^n \to N$, where $N$ is a $k$-dimensional manifold, there exists a fiber $f^{-1}(y)$ with the following property:
\[
\vol (\nu_t f^{-1}(y)) \ge \vol (\nu_t \mathbb RP^{n-k}),\ \forall t>0,
\]
the metric for $\mathbb RP^n$ is assumed coming from the projection from the round sphere $\mathbb S^n\to\mathbb RP^n$, and $\mathbb RP^{n-k}$ is a standard geodesic $(n-k)$-subspace in $\mathbb RP^n$.
\end{corollary}

Indeed, Corollary \ref{corollary:rpn-gromov} follows from the spherical waist inequality by considering the quotient
map from $\mathbb S^n$ to $\mathbb RP^n$. Note that the results~\cite{klartag2016,ak2016ball} are only valid in this limit case setting (for the Minkowski content), because the estimate for $n$-volume of a $t$-neighborhood in those works is not tight. 

Now let us think of examples of maps showing that the estimate of Corollary~\ref{corollary:rpn-gromov} is tight. Unlike the case of the sphere, building such examples is not obvious. The most symmetric way of having such is when we have a map $\mathbb RP^n\to N$ with \emph{all} fibers isometric to $\mathbb RP^{n-k}$. We know of several such examples:

\begin{theorem}
\label{theorem:rpn-hopfmaps}
The bound of Corollary~\ref{corollary:rpn-gromov} is tight for:
a) $n$ odd and $k=n-1$;
b) $n=7$ and $k=4$;
c) $n=15$ and $k=8$.
\end{theorem}

\begin{proof}
Case (a) follows from considering the usual projection  map $\mathbb S^n\to\mathbb CP^{(n-1)/2}$ which factors through $\mathbb RP^n\to\mathbb CP^{(n-1)/2}$, Cases (b) and (c) follow similarly from considering the Hopf maps $\mathbb S^7\to \mathbb S^4$ and $\mathbb S^{15}\to \mathbb S^8$ that also factor through $\mathbb RP^7$ and $\mathbb RP^{15}$ respectively.
\end{proof}

\begin{remark}
Considering maps $\mathbb RP^3\to N$ with $\dim N = 1$ we see that the bound in the Minkowski content version of Corollary~\ref{corollary:rpn-gromov} is not necessarily tight. This case is reduced to the case $N=\mathbb R$, otherwise $N$ is a circle $\mathbb R/\mathbb Z$, but the map evidently lifts to a map $f: \mathbb RP^3\to \mathbb R$ in this case, we find again Theorem~\ref{theorem:rp3functions}

%After this reduction we can improve the waist estimate of Corollary \ref{corollary:rpn-gromov} by the standard argument taking the median $m$ of the map $f$ and considering the set
%\[
%M = \{x\in\mathbb RP^3 : x\le m\}
%\]
%with $\vol_3 M = (1/2) \vol_3 \mathbb RP^3$. The isoperimetric inequality of Ritor\'e and Ros~\cite{rosrit1992} then implies that
%\[
%\vol_2 f^{-1}{m} \ge \pi^2 > \vol_2(\mathbb RP^2) = 2\pi
%\]
%and this bound is indeed attained by the composition of $\mathbb RP^3\to \mathbb CP^1$ with a certain function on $\mathbb CP^1$.
\end{remark}

\section{Complex projective space}

The method of~\cite{klartag2016,ak2016ball} also gives the following corollary from considering the projection $\mathbb S^{2n+1}\to \mathbb CP^n$:

\begin{corollary}
\label{corollary:cpn-gromov}
For any continuous map $f:\mathbb CP^n \to N$, where $N$ is a $2m$-dimensional manifold, there exists a fiber $f^{-1}(y)$ with the following property:
\[
\vol (\nu_t f^{-1} ) \ge \vol (\nu_t \mathbb CP^{n-m}),\ \forall t>0,
\]
where $\mathbb CP^{n-m}$ is a standard complex projective $(n-m)$-subspace in $\mathbb CP^n$.
\end{corollary}

\begin{proof}
We consider the projection $T : \mathbb S^{2n+1} \to \mathbb CP^n$, which takes the uniform measure on $\mathbb S^{2n+1}$ to a multiple of a uniform measure on $\mathbb CP^n$ and has the property that the distance between the circles $T^{-1}(x')$ and $T^{-1}(x'')$ equals the distance between $x'$ and $x''$ (this can be viewed as a normalization of the Riemannian metric on $\mathbb CP^n$, or the \emph{Riemannian submersion property} for $T$).

Applying Gromov's waist of the sphere theorem~\cite{grom2003,mem2009,karvol2013} to the composition $f\circ T$ then gives the result.
\end{proof}

\begin{remark}
Evidently, for any continuous map $f:\mathbb CP^n \to N$ with a manifold $N$ of odd dimension $k$, the same argument applies, but in this case we cannot find a natural submanifold like $\mathbb CP^{n-k/2}$ in $\mathbb CP^n$ to compare with.
\end{remark}

To establish upper bounds we again start with the question: Is it possible to foliate $\mathbb CP^n$ into $\mathbb CP^{n-k}$? Here is one case:

\begin{theorem}
\label{theorem:cpn-hopfmaps}
The bound of Corollary~\ref{corollary:cpn-gromov} is tight for:
b) $n=3$ and $m=2$;
c) $n=15$ and $m=4$.
\end{theorem}

\begin{proof}
This again follows from considering the Hopf maps $\mathbb S^7\to \mathbb S^4$ and $\mathbb S^{15}\to \mathbb S^8$, which factor through $\mathbb CP^3$ and $\mathbb CP^7$ respectively, because a quaternion or an octonion subspace is a complex subspace too.
\end{proof}

\section{Flat tori}
\label{section:tori}

Several observations for flat tori follow from the result of~\cite{klartag2016}. Let us fix some notation, for a full-dimensional lattice $\Lambda\subset\mathbb R^n$ we consider the torus $T_\Lambda = \mathbb R^n/\Lambda$ with the flat metric induced from $\mathbb R^n$.

We call a torus \emph{orthogonal} if $\Lambda$ is an orthogonal Cartesian product of one-dimensional lattices; if their covolumes are $a_1,\ldots, a_n$, we denote such an orthogonal torus by $T_{a_1,\ldots,a_n}$ and order the sizes $a_1\le \dots\le a_n$. From \cite[Corollary 9]{ak2016ball} 
%\cite[Theorem 1.4 and Corollary 5.3]{klartag2016} 
we infer:

\begin{corollary}
\label{corollary:tortho}
For any continuous map $f: T_{a_1,\ldots,a_n} \to N$, where $N$ is a $k$-dimensional manifold, there exists a fiber $f^{-1}(y)$ with
\[
\lmink_{n-k} f^{-1}(y) \ge \prod_{i=1}^{n-k} a_i.
\]
The bounds is attained at the natural projection $T_{a_1,\ldots,a_n} \to T_{a_{n-k+1},\ldots, a_n}$.
\end{corollary}

\begin{proof}
Put the open parallelotope $\prod_{i=1}(0, a_i)$ to the torus and apply \cite[Corollary 9]{ak2016ball} (in view of~\cite{karvol2013}).
\end{proof}

The case of a torus which is not orthogonal is less clear. An interesting question arises if we only allow maps to $\mathbb R^k$, not to arbitrary $k$-dimensional manifold. As was pointed to us by Isaac Mabillard, there exists a smooth map from the torus $T^n$ to $\mathbb R^n$ with multiplicity at most $2$, it is produced by embedding $T^{n-1}\to \mathbb R^n$ (e.g. by induction), then taking a tubular neighborhood of it and doubling the neighborhood, the double being diffeomorphic to $T^n$. Hence

\begin{theorem}
\label{theorem:tor2}
There exist maps $f: T_{a_1,\ldots,a_n} \to \mathbb R^k$ with $(n-k)$-volume of all fibers at most $2\prod_{i=1}^{n-k} a_i$.
\end{theorem}

\begin{proof}
Take the natural projection $T_{a_1,\ldots,a_n} \to T_{a_{n-k+1},\ldots, a_n}$ and follow by the map with multiplicity at most $2$ from $T_{a_1,\ldots,a_n}$ to $\mathbb R^k$.
\end{proof}

So, for maps from orthogonal tori to the Euclidean spaces, we know the optimal waist up to a factor of $2$. For the particular case $T_{a_1,\ldots,a_n} \to \mathbb R$ (continuous functions) the results are known since \cite{hadwiger1972}. Hadwiger's paper was in German, so we provide an argument here for completeness.

\begin{theorem}
\label{thm:torus-isoperimetry}
In the torus $T_{a_1,\ldots,a_n}$ any subset $M$ of half volume of the torus has
\[
\lmink_{n-1} \partial M \ge 2\prod_{i=1}^{n-1} a_i.
\]
The equality is attained at
\[
M = T_{a_1,\ldots,a_{n-1}}\times [0, a_n/2] \subset T_{a_1,\ldots,a_n}.
\]
\end{theorem}

\begin{proof}
Evidently, by a (cyclic) translation of the torus along its first coordinate $x_1$ it is possible to have
\[
\vol_n M\cap \{x_1\in (0, a_1/2)\} = \vol_n M\cap \{x_1\in (a_1/2, a_1)\}.
\]
Now we take the two parts $\{x_1\in (0, a_1/2)\}$, $\{x_1\in (a_1/2, a_1)\}$ and continue to translate them along the $x_2$ coordinates independently, then take the four parts and translate them independently along $x_3$ and so on.

In the end we will split the torus into $2^n$ open parallelotopes of size $a_1/2\times \dots \times a_n/2$ each having precisely half of its volume in $M$. Applying the isoperimetry for parallelotopes (see Lemma~\ref{lemma:isoperimetry-par}) we obtain that each of them has at least
\[
\prod_{i=1}^{n-1} (a_i/2) = 2^{-n+1} \prod_{i=1}^{n-1} a_i
\]
of $\lmink_{n-1} \partial M$, summing up we obtain the required estimate.
\end{proof}

The isoperimetry for a parallelotope follows, for example, from the isoperimetry of the Gaussian measure by the transportation trick of~\cite{klartag2016} (or see \cite{hadwiger1972}); let us provide a detailed argument in the form of the following:

\begin{lemma}
\label{lemma:isoperimetry-par}
Let a set $M\subset (0,a_1)\times\dots\times(0,a_n)$ have half volume of the parallelotope, then for the lower Minkowski content
\[
\lmink_{n-1} \partial M \ge \prod_{i=1}^{n-1} a_i,
\]
assuming $a_1\le a_2 \le \dots \le a_n$.
\end{lemma}

\begin{proof}
Construct the diffeomorphism $\tau : \mathbb R^n \to (0,1)^n$ as a Cartesian product of one-dimensional maps
\[
\tau_1 : \mathbb R \to (0,1)\quad x\mapsto \int_{-\infty}^x e^{-\pi t^2}\; dt.
\]
The map $\tau$ transports the Gaussian measure (call it $\gamma$) with density $e^{-\pi|x|^2}$ to the uniform density in the cube and is evidently $1$-Lipschitz.

Assume we have a set $N\subset (0,1)^n$ of half the measure with boundary $\partial N$. The set $\tau^{-1}(N)$ has half of the Gaussian measure and by the isoperimetric inequality (concentration property) for the Gaussian measure~\cite{sudakov-tsirelson1974}
\[
\gamma(\nu_t \tau^{-1}(\partial N)) \ge \int_{-t}^t e^{-\pi s^2}\; ds.
\]
From the Lipschitz property of $\tau$ we obtain
\[
\tau(\nu_t \tau^{-1}(\partial N) ) \subseteq \nu_t \partial N
\]
and therefore
\[
\vol_n (\nu_t \partial N) \ge \int_{-t}^t e^{-\pi s^2}\; ds.
\]

Now let us stretch the coordinate axes $A : (0,1)^n \to (0,a_1)\times\dots\times(0,a_n)$, assume $A(N) = M$. The set $\nu_t \partial N$ will get into $\nu_{a_nt} \partial M$ ($a_n$ is the largest of the $a_i$) and the volume gets multiplied by $a_1\dots a_n$, hence we have
\[
\vol_{n} ( \nu_{a_nt}\partial M ) \ge a_1\dots a_n \int_{-t}^t e^{-\pi s^2}\; ds.
\]
Replacing $a_n t$ by $u$ we have
\[
\vol_{n} ( \nu_u \partial M ) \ge a_1\dots a_n \int_{-u/a_n}^{u/a_n} e^{-\pi t^2}\; dt.
\]
Passing to the limit $u\to +0$ we obtain the estimate
\[
\lmink_{n-1} \partial M \ge a_1\dots a_n \lim_{u\to +0} \frac{\int_{-u/a_n}^{u/a_n} e^{-\pi t^2}\; dt}{2u} = a_1\dots a_n \frac{1}{a_n} = \prod_{i=1}^{n-1} a_i.
\]
\end{proof}

\section{Waists of convex bodies}
\label{section:convex}

Let us introduce a notation for the Minkowski content waist:
\[
\gamma_{n-k} (X) = \inf_{f : X \to \mathbb R^k \text{continuous}} \sup_{y\in \mathbb R^k} \lmink_{n-k} f^{-1}(y).
\]
%Of course, in this definition it is possible to consider real-analytic maps instead of continuous and arbitrary smooth $k$-dimensional manifolds instead of $\mathbb R^k$, $(n-k)$-Hausdorff measure instead of the Minkowski measure, etc. We have a certain isoperimetric property for this:

\begin{theorem}
\label{theorem:waistiso}
If $K\subset\mathbb R^n$ is a centrally symmetric convex body with $\vol_n K = v_n$ (the volume of the unit ball) then $\gamma_{n-k} (K) \le \gamma_{n-k} B^n = v_{n-k}$.
\end{theorem}

\begin{remark}
It is not clear if the same is true for not necessarily centrally symmetric convex $K$.
\end{remark}

\begin{proof}
In order to establish an upper bound for $\gamma_{n-k} (K)$ we are going to consider linear projections $\mathbb R^n\to\mathbb R^{n-k}$ with different $(n-k)$-dimensional kernels $L$. A corollary of the Brunn--Minkowski inequality tells that the function $\vol_{n-k} (L+v) \cap K$ is logarithmically concave in $v$, since it is also even in $v$, its maximum is attained at $v =0$. Therefore the waist for such a projection is just $\vol_{n-k} K\cap L$.

Now we use Zhang's result~\cite[Theorem 5.20]{koldobsky} that implies that from $\vol_n K = v_n$ it follows that for some $L$ we have $\vol_{n-k} K\cap L \le v_{n-k}$.
\end{proof}

\begin{remark} 
The waist of a convex body of volume $1$ can be arbitrary small. One might wonder which convex bodies of volume one minimize $max_{T\in SL(n)} \gamma_{n-k}(T(K))$ or which convex bodies in isotropic position $K$ minimize $\gamma_{n-k}(K)$. For both problems, the smallest waist that we know of is that of the cube.  
\end{remark}

Let us give a simple estimate for $\gamma_1(K)$ and then pass to its improved version for the Hausdorff measure:

\begin{theorem}[Bo'az Klartag, private communication]
\label{theorem:1waist-cs}
For centrally symmetric convex $K\subset\mathbb R^n$, the value $\gamma_1(K)$ is just the minimal width $w(K)$.
\end{theorem}

\begin{proof}
Inflate a centered ball $rB^n\subseteq K$ until it touches the boundary $\partial K$; this will occur at (possibly not unique) pair of points $\{x, -x\}\subset\partial K$. Evidently, $2r$ will be the width of $K$ in this situation.

At any of these points the support hyperplane to $rB^n$ is also a support hyperplane to $K$, hence $K$ is between the hyperplanes with distance $2r$ between them. So the projection whose kernel is the line
in direction $x$ establishes
\[
\gamma_1(K)\le 2r.
\]
Also, the waists are inclusion-monotone and therefore
\[
\gamma_1(K)\ge \gamma_1(B) = 2r.
\]
\end{proof}

As previously mentioned, the $1$-Hausdorff measure of a compact set is at most its $1$-dimensional lower Minkowski content, see~\cite[Theorem 1.8]{ds1993}. Therefore Theorem \ref{theorem:1waist-cs} follows from the following:

\begin{theorem}
\label{theorem:1waist}
For any convex body $K\subset\mathbb R^n$ and a continuous map $f : K\to \mathbb R^{n-1}$ there exists a fiber $f^{-1}(y)$ of $1$-Hausdorff measure at least $w(K)$. The estimate is tight for any $K$, in particular $\gamma_1(K) = w(K)$.
\end{theorem}

\begin{proof}
By the Alexandrov waist theorem~\cite[Theorem 6.2]{kar2012} there exists a connected $X\subset K$ such that $f(X)$ is a single point $y$ and $X$ cannot be covered by a smaller homothet of $K$. This $X$ is a connected component of the fiber $f^{-1}(y)$.

Let us first assume that the map $f$ is real-analytic and the fibers of the map are decent one-dimensional sets with length defined in any reasonable way; let the length of $X$ be $\ell$.

The argument is very similar to~\cite[Lemma 5.5]{akp2014}. Assume $\ell < w(K)$ and choose $\epsilon>0$. Split $X$ into small smooth curve segments $I_1,\ldots,I_N$ so that on every such segment the unit tangent vector varies by at most $\epsilon$. The difference set $K - K$ has width $2w(K)$ and, as in the proof of Theorem~\ref{theorem:1waist-cs}, it contains the ball $w(K) B^n$; therefore $K$ contains a line-segment of length $w(K)$ in any given direction. Hence it is possible to cover every $I_i$ by a homothet $\alpha_i K + t_i$ so that
\[
\alpha_i w(K) \le \ell(I_i) (1+\epsilon).
\]
If two such homothets with factors $\alpha_i$ and $\alpha_j$ intersect then they can be together covered by a single homothet with factor $\alpha_i+\alpha_j$. From the connectedness of $X$ we see that this process of merging will end up when $X$ will be covered by a single homothet $\alpha K + t$ with
\[
\alpha w(K) \le \ell (1 + \epsilon)\ \Rightarrow\ \alpha \le \frac{\ell (1 + \epsilon)}{w(K)}.
\]
Choosing a sufficiently small $\epsilon$ and using the assumption $\ell < w(K)$ we obtain a contradiction with the choice of $X$.

Now pass to the case of a less regular component of the fiber $X$, this is just a compact connected set. Assume
\[
\haus_1 (X) = \ell < W < w(K).
\]
By the standard application of Helly's theorem there exists a subset of at most $n+1$ points $L\subset X$ such that $L$ already does not fit into a smaller homothet of $K$.  Invoke the definition of the $1$-Hausdorff measure, choose $\delta>0$ and cover $X$ with convex compacta $X_1,\ldots, X_N$ such that each has diameter $<\delta$ and the sum of their diameters is $<W$.

Since $X$ is connected then the intersections graph of the family $\{X_i\}$ is connected. Choose a minimal subgraph $T$ spanning the vertices touching $L$; this will be a tree with at most $n-1$ branchings (vertices of degree $>2$). Now for any edge $e\in T$, corresponding to an intersecting pair $X_i\cap X_j$, mark a point $p_e\in X_i\cap X_j$. Also mark all the points in $L$. Connect the marked points that belong to the same $X_i$ by straight line segments to obtain a connected graph $G$ spanning all marked points.

The graph $G$ has length at most $W+n\delta$, because its length is just the sum of diameters of the respective $\{X_i\}$ plus an error term that arises from the branching. Choosing $\delta$ sufficiently small we will obtain a connected graph $G$ of length $<w(K)$; but since $G$ contains $L$ by the construction, it cannot be covered by a smaller homothet of $K$. We have reduced the problem to a regular set $G$ in place of $X$, considered above.
\end{proof}

\begin{remark} 
Write $B_p^n = \{ x \in \mathbb R^n \, ; \, \sum_{i=1}^n |x_i|^p \leq r_{p,n} \}$, where $r_{p,n} > 0$ is determined by the requirement that $\vol_n(B_p^n) = 1$. Proposition 5.21 in \cite{lo2008} supplies an explicit transport map $\tilde{S}_{p,n}: \mathbb R^n \rightarrow \mathbb R^n$ which pushes forward the standard Gaussian measure on $\mathbb R^n$ to the uniform measure on $B_p^n$. When $p \geq 2$, the Lipschitz constant of $\tilde{S}_{p,n}$ is bounded by the constant $18$. We may therefore deduce bounds for the waist of $B_p^n$ from the Gaussian waist inequality. This provides a result similar to Theorem 1.3 from \cite{klartag2016}, for the particular case where $K = B_p^n$ and $p \geq 2$. 
\end{remark}

\section{Waists of sweepouts}
\label{section:sweepouts}

\subsection{The fundamental class.}

In~\cite{grom1983} Gromov (referring to Almgren) describes another version of the waist inequality, arising from the Lyusternik--Schnirelmann-type theory for the space of cycles: For a $k$-dimensional family $\mathcal F$ of rectifiable $(n-k)$-dimensional cycles in the round sphere $\mathbb S^n$, passing through a generic point of $\mathbb S^n$ odd number of times, some cycle $c\in \mathcal F$ has $(n-k)$-volume greater or equal to the volume of the equatorial $\mathbb S^{n-k}\subset \mathbb S^n$. 

Here a \emph{$k$-dimensional family of cycles} is understood as a singular $k$-cycle in the space of cycles and the phrase \emph{passing through a generic point of $\mathbb S^n$ an odd number of times} is a heuristic which we will refer to a \emph{sweepout} of $\mathbb S^n$. Formally we assume that $\mathcal F$ detects the fundamental class $a(n-k, \mathbb S^n)$ in the cohomology of the space of $(n-k)$-cycles. More details about this will be given in Subsection~\ref{section:cupproducts} below. The difference with the waist inequalities we have discussed so far is that two cycles of a sweepout $F$ might intersect while two fibers of a map $f$ cannot.

Let us show how to obtain similar Gromov--Almgren type results for other spaces with the $1$-Lipschitz technique of~\cite{klartag2016,ak2016ball}. Consider the orthogonal projection 
\[
p_m : \mathbb S^{n+m}\rightarrow B^n,
\] 
where $\mathbb S^{n+m}$ and $B^n$ are considered as the unit $(n+m)$-sphere and the unit $n$-dimensional ball centered at the origin in $\mathbb{R}^{n+m+1}$. Let $s_i$ denote a volume of $\mathbb S^i$ and $v_i$ the volume of $B^i$, we know that $s_i = (i+1) b_{i+1}$.

The uniform Riemannian measure $\sigma$ of $\mathbb S^{n+m}$ is mapped to a measure $\mu_m$ in $B^n$ with density 
\[
s_m(1-|x|^2)^{\frac{m-1}{2}},
\]
the scaling factor can be seen working in a neighborhood of the fiber $p_m^{-1}(0)$. Note that 
\[
\int \limits_{B^n} d\mu_m=s_{n+m}.
\]

Since $p_m$ is a $1$-Lipschitz map, we have for all $t>0$
\begin{equation}
\label{equation:projinclusion}
p_m^{-1}( \nu_t X )\supseteq \nu_t p_m^{-1}(X),
\end{equation}
where $\nu_t$ denotes the $t$-neighborghod of the set in $\mathbb S^{n+m}$. Therefore for any $X\subset B_n$ we have 
\begin{equation}
\label{equation:projinequality}
\mu_m( \nu_t X ) \ge \sigma( \nu_t p^{-1}(X) ).
\end{equation}

Note that if $X$ a rectifiable $(n-k)$-dimensional subset then $p_m^{-1}(X)$ is a rectifiable subset of the sphere with dimension $n+m-k$. We can then calculate the (weighted) volumes as the Minkowski content by~\cite[3.2.39]{federer1969}
\begin{equation}
\label{equation:fiberinequality}
\vol_{n-k}^{\mu_m} X
=
\lim_{t \to +0} \frac{\mu_m (\nu_t X )}{v_k t^k}
\geq
\lim_{t \to +0} \frac{\sigma(\nu_t p_m^{-1}(X) )}{v_k t^k}=
\vol_{n+m-k}p_m^{-1}(X),
\end{equation}
for $X$ we use its weighted volume with the density of the measure $\mu_m$ as the weight.

In the case $X$ is an $(n-k)$-dimensional linear subspace, we obtain the equality in (\ref{equation:fiberinequality}). Indeed, in this case the directions orthogonal to $p_m^{-1}X$ are projected by the differential of $p_m$ isometrically, hence we will have an approximate inverse of (\ref{equation:projinclusion}):
\[
p_m^{-1}( \nu_{t - o(t)} X )\subseteq \nu_t p_m^{-1}(X).
\]

Going to the limit $m\to\infty$ and scaling the $\mathbb R^n$ we may consider a Gaussian measure $\gamma$ (say, with density $\rho = e^{-\pi |x|^2}$) as the limit of such $\mu_m$. This can also be written explicitly. We have a statement for the projected density $\rho_m = s_m(1-|x|^2)^{\frac{m-1}{2}}$, and we want to scale $h_m : \mathbb R^n \to \mathbb R^n$ with
\[
y = h_m(x) = \sqrt{\frac{m-1}{2\pi}} x
\]
so that the density becomes
\[
\rho'_m = \left(\frac{2\pi}{m-1}\right)^{n/2} s_m \left(1 - \frac{2\pi}{m-1}|y|^2 \right)^{\frac{m-1}{2}}
\]
for $|y|\le \sqrt{\frac{m-1}{2\pi}}$ and $\rho'_m = 0$ otherwise. For this density we know that if $X$ of codimension $k$ has $(n-k)$-dimensional $\rho'_m$-weighted volume at most that of an $(n-k)$-dimensional linear subspace in $\mathbb R^n$ then its preimage $(h_m\circ p_m)^{-1}(X)$ has $(n+m-k)$-volume at most that of an equatorial $\mathbb S^{n+m-k}\subset \mathbb S^{n+m}$. Stated this way, this property does not depend on multiplying $\rho'_m$ by a constant, so we may take eventually
\[
\rho''_m = \left(1 - \frac{2\pi}{m-1}|y|^2 \right)^{\frac{m-1}{2}}
\]
for $|y|\le \sqrt{\frac{m-1}{2\pi}}$ and $\rho'_m = 0$ otherwise. As $m\to\infty$, this density converges to $\rho = e^{-\pi|y|^2}$ in a monotone increasing fashion. For this normalized density, the $(n-k)$-volume of any $(n-k)$-linear subspace is always $1$.

Now assume we have a family of cycles $\{X\}_{X\in\mathcal F}$ in $\mathbb R^n$ with all $(n-k)$-dimensional $\rho$-weighted volume less or equal to $1-\epsilon$, the same (maybe with smaller $\epsilon$) is true for sufficiently large $m$ and the measure $\rho''_m$. In order to deal with cycles in $\mathbb R^n$ we just compactify it with one point at infinity and use the cycles if thus obtained a topological sphere $S^n = \mathbb R^n\cup\{\infty\}$. From the (\ref{equation:fiberinequality}) and the explanations above we see that every cycle in the family 
\[
\mathcal F_m = (h_m\circ p_m)^{-1}\mathcal F = \{(h_m\circ p_m)^{-1}(X)\}_{X\in\mathcal F}
\]
will then have $(n+m-k)$-dimensional volume strictly less than $s_{n+m-k}$. If $\mathcal F$ passes through a generic point of $\mathbb R^n$ odd number of times then it is geometrically evident that $\mathcal F_m$ passes through a generic point of $\mathbb S^{n+m}$ odd number of times. Thus the Gromov--Almgren result for $\mathcal F_m$ applies and we obtain:

\begin{theorem}
For a $k$-dimensional family $\mathcal F$ of rectifiable $(n-k)$-dimensional cycles sweeping out $\mathbb R^n$,  some cycle $c\in \mathcal F$ has $e^{-\pi|x|^2}$ weighted $(n-k)$-volume greater or equal to $1$.
\end{theorem}

Applying the construction of~\cite{klartag2016} we see:

\begin{theorem}
For a $k$-dimensional family $\mathcal F$ of rectifiable $(n-k)$-dimensional cycles sweeping out $Q=[0,1]^n$ (the cycles are considered relative to $\partial Q$), some cycle $c\in \mathcal F$ has $(n-k)$-volume greater or equal to $1$.
\end{theorem}

\begin{theorem}
For a $k$-dimensional family $\mathcal F$ of rectifiable $(n-k)$-dimensional cycles sweeping out $T^n=\mathbb R^n/\mathbb Z^n$, some cycle $c\in \mathcal F$ has $(n-k)$-volume greater or equal to $1$.
\end{theorem}

The next theorem is obtained without going to the limit, just putting $m=1$:

\begin{theorem}
For a $k$-dimensional family $\mathcal F$ of rectifiable $(n-k)$-dimensional cycles sweeping out the unit ball $B^n$ (the cycles are considered relative to $\partial B$),  some cycle $c\in \mathcal F$ has $(n-k)$-volume greater or equal to $v_{n-k}$.
\end{theorem}

Evidently, all the estimates in the above theorems are tight for $k$-dimensional families of translation of an $(n-k)$-dimensional linear subspace.

\subsection{Volume of cup products}
\label{section:cupproducts}

Let us discuss the volume of cohomology classes in the space of cycles of the cube, we follow section 2 of Guth's paper \cite{guth2008} where the reader can find more details. We recommend  the introduction of that paper to build intuition on the space of cycles.  To define $\mathcal Z_{n-k}([0,1]^n, \partial[0,1]^n)$, the space of relative $(n-k)$-cycles of the cube, let $C_{\ell}(M)$ be the set of formal sums with $\mathbb{Z}_2$ coefficients $\sum_{a_i\in \mathbb{Z}_2} a_i f_i(\sigma_{\ell})$ of Lipshitz maps $f_i\colon \sigma_{\ell} \to M$ from the $\ell$-dimensional simplex $\sigma_{\ell}$ to a space $M$. Let $\partial$ be the usual boundary operator from singular homology. A chain $z \in C_{n-k}([0,1]^n)$ is a relative cycle if $\partial z \in C_{n-k-1}(\partial[0,1]^n)$. Given two homologous relative cycles $z_1$ and $z_2$, their flat distance is the smallest volume of a chain testifying that they are homologous\footnote{The general definition of the flat norm of a chain is more complicated but Guth \cite[Appendix A]{guth2008} proves the equivalence with this version for the case of cycles.}, i.e.,

\begin{alignat}{2}
d(z_1,z_2) &:=\inf_{\substack{\partial c=z_1+z_2+b}
    }& \vol(c),
\end{alignat}  
where $c$ is an element of $C_{n-k+1}([0,1]^n)$, and $b$ is an element of $C_{n-k}(\partial [0,1]^n)$. The space of cycles $\mathcal Z_{n-k}([0,1]^n, \partial[0,1]^n)$ is the completion of this metric space after identifying any two cycles that are at distance $0$.

By a \emph{family of cycles} we mean a continuous map from a simplicial complex $Y$ to the space of relative cycles $\mathcal Z_{n-k}([0,1]^n, \partial[0,1]^n)$. Almgren showed that a family of cycles 
\[
F \colon Y \to \mathcal Z_{n-k}([0,1]^n, \partial[0,1]^n)
\]
 induces well defined homomorphisms $H_i(Y)\to H_{i+n-k}([0,1]^n,\partial[0,1]^n)$, the so called gluing homomorphisms. The idea of building such homomorphisms is to pick a fine triangulation of $Y$ and perturb $F$ slightly to obtain a map $G$ defined on the $0$ skeleton of the triangulation, that takes each vertex to a $(n-k)$-Lipschitz cycle. Then extend this map to the whole triangulation, working inductively in the dimension. More precisely, the isoperimetric inequality ensures that  $d(F(y),G(y))$ is small for all $y \in Y$.  The continuity of $F$ and the assumption on the triangulation imply that neighbouring vertices in the triangulation $x$ and $y$, are mapped to cycles $G(y)$ and $G(x)$ which bound a chain $c$ of small volume i.e. $\partial c= G(x)+G(y)$. Then we extend $G$ letting $G(e):=c$, where $e$ is the edge $(x,y)$. 
 
By construction, $G$ is a chain map from the $1$-skeleton of the triangulation of $Y$ to a $n-k+1$ complex. For every $2$-face $\sigma$ in the triangulation of $Y$, its boundary $\partial \sigma$ is mapped by $G$ to a $(n-k+1)$-relative cycle in $([0,1]^n,\partial[0,1]^n)$ of small volume. By the isoperimetric inequality there exists a Lipschitz $(n-k+2)$-dimensional relative chain $c'$ of small volume such that $\partial c' = G(\partial \sigma)$ so we define $G(\sigma):=c'$. In general, working dimension up ensures that each $i$-face of the triangulation of $Y$ maps to a Lipschitz $(n-k+i)$-dimensional relative chain in $[0,1]^n$ so that boundary maps commute. Since boundary maps commute, the map $G$ defines chain maps $C_i(Y)\to C_{i+n-k}([0,1]^n,\partial[0,1]^n)$ that factor through $C_i(\mathcal Z_{n-k}([0,1]^n, \partial[0,1]^n))$ and induce homomorphisms 
\[
G_* : H_i(\mathcal Z_{n-k}([0,1]^n, \partial[0,1]^n))\to H_{i+n-k}([0,1]^n,\partial[0,1]^n.
\] 

If the triangulation was chosen sufficiently fine, the isoperimetric inequality can be used to show that any two such discretizations $G_1$, $G_2$ (gluings maps) of $F$ give rise to chain homotopic maps, $G_1\sim G_2$, so the corresponding homology homomorphisms depend only on $F$. The case $i=k$ corresponds to $H_{n}([0,1]^n,\partial[0,1]^n)=\mathbb{Z}_2$, so we have a map 
\[
G_*\colon H_k(\mathcal Z_{n-k}([0,1]^n, \partial[0,1]^n)) \to \mathbb{Z}_2,
\]
which by the universal coefficients theorem can be identified with a class in $a(n-k,n) \in H^k( \mathcal Z_{n-k}([0,1]^n, \partial[0,1]^n) )$. Let us describe this class a little bit more: A $k$-dimensional family of $(n-k)$-dimensional cycles $F : Y \to \mathcal Z_{n-k}([0,1]^n, \partial[0,1]^n)$ detects the class $a(n-k,n)$ if, after gluing $F$ into an $n$-cycle, we obtain a nontrivial homology class in $H_{n}([0,1]^n,\partial[0,1]^n)=\mathbb{Z}_2$. In other words, we say that the family of cycles $F$ \emph{is a sweepout} if $F^*a(n-k,n) \neq 0$. 

More generally, given a singular cohomology class $\alpha \in H^i(\mathcal Z_{n-k}([0,1]^n, \partial[0,1]^n))$, we will say that a family of cycles $F$ detects $\alpha$ if $F^*\alpha\neq 0 \in H^*(Y)$. We define the waist of a family as the largest of the volumes in the image. We define the volume of a class $\alpha$, as the infimum of the waist among all families $F$ that detect $\alpha$:
 
\begin{alignat}{2}
\vol(\alpha) &:=\inf_{\substack{F\colon X \to \mathcal Z_{n-k}([0,1]^n, \partial[0,1]^n) \\ F^*(\alpha)\neq 0}}\sup_{x\in X} \vol(F(x)).
\end{alignat}

This construction makes rigorous the concept of sweepout by a family of cycles on the previous section. Moreover the fibers of a smooth function $f:[0,1]^n\to \mathbb{R}^k$ detect the fundamental class, where we set $Y=\mathbb R^k\cup\{\infty\}$ to be the one point compactification of $\mathbb R^k$ and the fiber over $\infty$ is a zero cycle. So the waist inequality in terms of families of cycles is more general then the waist inequality in terms of (sufficiently regular) functions. 

Now we are ready to discuss the volume of the cup product class $a(n-k,n)^p$. Heuristically, a family of relative cycles $F$ detects $a(n-k,n)^p$ if for a generic $p$-tuple of points $x_1,x_2, \ldots x_p \in [0,1]^n$, there exists a cycle $z \in F$ that passes an odd number of times through each $x_i$.  The explanation behind this heuristic lies on the following result from Lyusternik--Schnirelmann theory: If classes $\alpha$ and $\beta$ vanish on open sets $A$ and $B$ respectively, then their cup product $\alpha \cup \beta$ vanishes on $A \cup B$. We would like to apply this lemma to open sets in the space of cycles. Let $U$ correspond to a regular enough subdomain of the cube $[0,1]^n$, we apply the Lyusternik--Schnirelmann lemma to the inverse image of the open interval $(0,\gamma_{k}(U))$ under the volume functional as follows. 

If the volume was a continuous real valued function on the space of cycles with the flat topology, then $\{z \in \mathcal Z_{n-k}(U, \partial U): \vol(z)<\gamma_{k}(U)\}$ would be an open set in the space of cycles. Choosing disjoint domains $U,V$ the cup product class $a(n-k,n)^2$ would vanish on the union of the corresponding subsets of the space of cycles 
\[
\{z \in \mathcal Z_{n-k}(U, \partial U): \vol(z)<\gamma_{k}(U)\} \cup \{z \in \mathcal Z_{n-k}(V, \partial V): \vol(z)<\gamma_{k}(V)\},
\]
and therefore it will also vanish on  $\mathcal Z_{n-k}([0,1]^n, \partial[0,1]^n)$, which would yield a contradiction, implying that the volume of $a(n-k,n)^2$ is larger than $\gamma_{k}(V)+\gamma_{k}(U)$. However the volume function is lower semicontinuous but is not continuous in the space of cycles with the flat norm topology. For example, a closed curve of large perimeter that bounds a very thin region is very close to having $0$ flat norm, but has large $1$-volume.  Gromov and Guth deal with this issue using the filling volume instead of the volume. To go back to the volume they apply the isoperimetric inequality. Very recently \cite{conelio2016} Liokumovich, Coda-Marques, and Neves found an alternative way to deal with the continuity issue. Following Pitts, they use the mass topology on the space of cycles which makes the volume into a continuous function. Roughly, Corollary 2.13 of \cite{conelio2016} claims that the family of cycles that minimizes the waist of the families detecting $a(n-k,n)^p$ can be assumed to be continuous in the mass-topology. Beyond ideas that extend previous results of the Almgren--Pitts theory, the cup product context requires a technical lemma to deal with the restriction of a family of cycles to a subdomain (Lemma 2.15 of \cite{conelio2016}).

Our interest here is recording the best values we could obtain for the constants $c_{n,k}, C_{n,k}>0$  such that $C_{n,k}\geq \frac{\vol[a(n-k,n)^p]}{p^{k/n}}\geq c_{n,k}$. The existence of these constants was shown in \cite{grom2003,guth2008}. Our contribution consists in stratforwardly combining the recent results from \cite{klartag2016, conelio2016} with the original idea of \cite{grom2003} for the lower bound, and a minor variation on the arguments of \cite{guth2008, guth2006} for the upper bound.
 
\begin{theorem}
\label{theorem:cup-power}
If $p$ is of the form $\ell^n$, and $a(n-k,n)$ is the fundamental class of the space of relative cycles of the cube, then 
\[
2^{n+k} \binom{n}{k} p^{k/n}\geq \vol[a(n-k,n)^p]\geq p^{k/n}.
\]
\end{theorem}

In the case $k=1$, the upper bound can be improved, for $p$ large with respect to $n$, $n^2 p^{1/n} \geq \vol[a(n-1,n)^p]$. Guth conjectures that the volumes of cup product classes $a(n-1,n)^p$ of the round ball relative to its boundary are achieved by families of zeros of $n$ variable polynomials. He shows that if $p=\binom{d+n}{d}-1$ then the family of polynomials of degree $d$ detects this class. For the cube a similar reasoning gives the claimed improvement. The argument goes back at least to Gromov \cite{grom2006} and is as follows (adapted to the cube), by Crofton's formula the surface area of a hypersurface is proportional to the expected number of intersections with a random line. Any line that intersects the interior of $[0,1]^n$, intersects its boundary exactly twice and intersects a curve of degree $d$ at most $d$ times, it follows that the $(n-1)$-dimensional volume of the intersection of such a curve with the cube is at most $\frac{d}{2} 2n$, where the $2n$ term corresponds to the surface area of the boundary of the cube. Unlike Guth's bend construction (explained below) that gives the result up to factor of $c^n$, these families have small volume also as cycles with integer coefficients. A careful computation for the case $k=1$, $n=2$ gives $\vol[a(1,2)^p]\leq 2\sqrt{2} p^{1/2}$. We begin with the proof of the lower bound, where we might pretend that the volume is continuous in the space of cycles as previously discussed based on the results of \cite{conelio2016}\footnote{In \cite{conelio2016} the continuity is used to confirm a conjecture of Gromov: $\frac{\vol[a(n-1,n)^p]}{p^{1/n}}$ converges to an absolute constant as $p$ goes to infinity. This holds for any Riemannian manifold, and the constant depends only on the volume of the manifold, which is an analogue to the celebrated Weyl law.}

\begin{proof}[Proof of the lower bound in Theorem~\ref{theorem:cup-power}]
Partition $[0,1]^n$ into $p=\ell^n$ equal cubes. Given a family of cycles $F$, consider for each cube $Q_i$, the subfamily of cycles $S(i):=\{x \in X: \vol(F(x) \cap Q_i)<(1-\epsilon) \vol(Q_i)^{\frac{n-k}{n}}\}$. By the waist inequality $S(i)$ does not detect the fundamental class of $Q_i$. By the vanishing lemma of LS-theory, the cohomology class $F^*(a(n-k,n))^p$ vanishes on $\cup_{i=1}^p S(i)$. Since by assumption $F$ detects $a(n-k,n)^p$ then there is a cycle $z_0$ not in $\cup_{i=1}^p S(i)$ but in $F$, that for each $i$, $\vol(z\cap Q_i)\geq (1-\epsilon) \vol(Q_i)^{\frac{n-k}{n}}$. Summing up the $p$ cubes and taking the supremum provides the estimate.
\end{proof}

Let us explain the upper bounds for the case $a(1,2)^p$, this is an essential particular case with a clear geometric picture. Consider the family of relative cycles $F_\alpha$ that consist of all $p$ parallel lines with slope $\alpha$, this family detects $a(1,2)^p$, since for any $x_1,x_2, \ldots x_p$, there is a set of $p$ lines with slope $\alpha$ such that each point $x_i$ is contained in one line of this set. 

The next step is to deform this family so that each relative cycle is a sum of $p$ relative cycles (the images of the lines) that overlap a lot. More precisely, we consider $\Psi(F_\alpha)$ where  $\Psi\colon ([0,1]^2, \partial[0,1]^2)\to ([0,1]^2, \partial[0,1]^2)$ is a degree one map fixing the boundary. A linear interpolation between the identity and $\Psi$ shows that they are homotopy equivalent, from which it follows that $\Psi(F_\alpha)$ also detects $a(1,2)^p$. The map $\Psi$ will push almost all of the unit cube to the $1$ skeleton of a scaled lattice. Specifically, subdivide the cube as before using the grid of side length $p^{-1/2}$. Let $Q_i$ be a small cube as before and let $\epsilon Q_i$ be a homothet of $Q_i$ with the same center and side length $\epsilon p^{-1/2}$. Consider an irrational slope $\alpha$, and choose $\epsilon>0$ so that no line with slope $\alpha$ intersects two homothets $\epsilon [0,1]^2$ now define $\psi$ on each sub square $Q_i$. On $Q_i-\epsilon Q_i$ consider the map that pushes radially from the center of $Q_i$ to $\partial Q_i$, on $\epsilon Q_i$, consider the homothety $\epsilon Q_i \to Q_i$. 

Let us now estimate the length of the image under $\Psi$ of a $p$-tuple of lines with slope $\alpha$. We denote this relative cycle by $z$ and write it as $z=z_1+z_2$ where $z_1$ is supported in the $1$-skeleton of the grid, and $z_2$ is supported on the union of the interiors of the squares $Q_i$. Since we are dealing with $\mathbb{Z}_2$ cycles the length of $z_1$ is smaller than the length of the $1$-skeleton of the grid, which is $2(p+p^{1/2})p^{-1/2}=2p^{1/2}+2$. On the other hand by the choice of the slope and $\epsilon$ each line of a family contributes at most $2p^{-1/2}$ to the length of $z_2$ so the total length of $z_2$ is bounded above by $2p^{1/2}$. In total we have $\vol[a(1,2)^p]\leq 4 p^{1/2}+2$.

The previous argument generalises to higher dimensions. We again refer the reader to Guth's papers \cite{guth2008, guth2006} for more details, and recall enough to analyze the constant.

\begin{proof}[Proof of the upper bound in Theorem~\ref{theorem:cup-power}] 
Let $P$ be a $k$-flat, let $\pr_P$ be the orthogonal projection onto $P$. Consider the set of $p$-tuples of points $x_1,x_2\ldots x_p \in P$. The corresponding set of $p$-tuples of paralllel flats $\pr_P^{-1}(x_1), \pr_P^{-1}(x_2), \ldots, \pr_P^{-1}(x_p)$, is a family of relative cycles that detects $a(n-k,n)^p$. The family that we will construct is the image of $F_P$ under a map $\Psi_\epsilon$. In pages 33-35 of \cite{guth2008} Guth shows that the family of relative cycles $F_P$ that consists of all sets of at most $p$ $(n-k)$-flat perpendicular to a given $k$-flat detects the cohomology class $a(n-k,n)^p$. The idea is build a class that detects $a(0,k)$ and iterate $n-k$ suspensions that can be realised by a product of each $0$ cycles with $\mathbb R^k$. \\  Again we will deform the family $F_P$ applying a degree $1$ transformation from the cube to itself, that leaves the boundary fixed and is homotopic to the identity.  We denote by $\Gamma$ the integer grid scaled to have side length $p^{-1/n}$.  Abusing notation we denote also by $\Gamma$ the induced cell structure. The transformation pushes almost all of the cube to the $(n-k)$-skeleton of $\Gamma$. Let $\Gamma^*:=\Gamma+ (\frac{1}{2},\frac{1}{2}\ldots \frac{1}{2})$ and $\skel_{k-1}(\Gamma^*)$ its $k-1$ skeleton. Define $\Psi\colon [0,1]^n \to [0,1]^n$ so that the complement of $\nu_\epsilon \skel_{k-1}(\Gamma^*)$ is mapped to $\skel_{n-k}(\Gamma)$ and $\nu_\epsilon \skel_{k-1}(\Gamma^*)$ is mapped to $[0,1]^n$. Specifically, let $\phi$ be a $(n-k)$-face of $\Gamma$, let $\phi^*$ be the unique $k$-face of $\Gamma^*$ that is perpendicular and transversal to $\phi$. To ease notation put $l=p^{-1/n}$. After applying a global isometry $\phi$ can be assumed to be, \[\{x:0 \leq x_i \leq l \textrm{ for }1 \leq i \leq n-k \textrm{ and } x_i=0 \textrm{ for } n-k \leq i \leq n\},\] and a facet $\eta$ of $\partial \phi^*$, then can be assumed to be, \[\{x: x_i = \frac{l}{2} \textrm{ for } 1\leq i \leq n-k+1 \textrm{ and }-\frac{l}{2} \leq  x_i\leq \frac{l}{2} \textrm{ for } n-k+1 \leq i \leq n\}.\]

Now consider the join \[K=\phi*\eta=\{(1-t)x+t y: t \in [0,1], x \in \phi \textrm{ and } y \in \eta\}.\] This is a convex polyhedron, as we let $\phi$ vary over $\skel_{n-k}(\Gamma)$ and consider the join with the corresponding $2k$ facets of $\phi^*$ we obtain a tiling of space by tiles isometric to the one we just described in coordinates. To define $\Psi_\epsilon$ we define it on each tile of the tiling so we might use barycentric coordinates for $K$ as before. 
\[   
\Psi_\epsilon((1-t)x+t y)= 
     \begin{cases}
       x & \text{ if } t>\epsilon\\
      (1-\frac{t}{\epsilon}) x+  \frac{t}{\epsilon}y & \text{ if } t\leq \epsilon .\ 
     \end{cases}
\]
Notice that this is well defined because $\phi$ and $\eta$ sit on disjoint, perpendicular affine spaces whose dimensions sum to $(n-1)$.This defines a map on one tile, it is easy to check maps neighboring tiles piece together into a global map from $\mathbb R^n$ to $\mathbb R^n$ that restctricts to a map from the cube to the cube. 
Varying $\epsilon>0$ provides a homotopy equivalence between $\Psi_\epsilon$ and the identity.  Let $z$ be a cycle in the family $\Psi_\epsilon(F_P)$, we cut it into two pieces $z_1+z_2$, where $z_1$ is supported in $\skel_{n-k}(\Gamma)$ and $z_2$ in the complement. Since we are working with modulo $2$ coefficients if $\vol (z_1)$ is finite then $\vol(z_1)\leq p\binom{n}{k} p^{\frac{k-n}{n}}=\binom{n}{k} p^{k/n}$.  \\
The cycle $z_2$ is the image under $\Psi_\epsilon$ of the intersection between $\nu_\epsilon \Gamma^*$ and $\pr_P^{-1}(x_1) \cup  \pr_P^{-1}(x_2) \cup  \ldots \cup  \pr_P{-1}(x_p)$. We begin bounding the number of connected components of this intersection. Assume again that $P$ is in general position with respect to $\Gamma$, and consider the projection $\pr_P[\skel_{k-1}(\Gamma^*)]$. It is a finite union of $(k-1)$-flats in general position in the $k$-flat $P$, therefore at most $k$ of them have a non empty intersection.  we can conclude that the $(n-k)$-flat $\pr_P^{-1}(x_i)$ intersects  $\Gamma^*$, at most $k$ times. Choosing a sufficiently small $\epsilon>0$, we obtain that $\pr_P^{-1}(x_i) \cap [\nu_\epsilon \skel_{k-1}(\Gamma^*)]$ has at most $k$ connected components. We fix $x_i$ for the rest of the proof. In one connected component of $\pr_P^{-1}(x_i)\cap (\nu_\epsilon \Gamma^*)$, $\pr_P^{-1}(x_i)$ intersects the interior of at most $2^n-1$ tiles isometric to $K=\phi*\eta$. Now let $\dot K$ be the interior of $K$, we claim that  $\vol_{n-k}(\Psi_\epsilon [(\pr_P^{-1}(x_i)]\cap \dot K) \leq 2^{k} \vol_{n-k}(\phi)$. The same estimate holds for other tiles.  The set $\Psi_\epsilon [(\pr_P^{-1}(x_i)] \cap \dot K$ can be described as,  \[V:=\{((1-t)x+ty) \in \dot K, \textrm{ such that for } v_1,v_2\ldots v_{k},  \langle (1-\epsilon t) x+ \epsilon t y)\cdot v_j \rangle=1\}.\]
Where the collection of vectors $\{v_j\}$ is such that $\pr_P^{-1}(x_i)=\cap_{j=1}^k \{z\in \mathbb R^n: z\cdot v_j=1\}$. Rewritting these conditions in cartesian coordinates it is clear that each vector $v_j$ defines a quadratic conditions. Therefore a $k$-dimensional affine space can intersect $V$ in at most $2^k$ points, while it intersects $\phi$ at most once, applying Cauchy-Crofton's formula now yields the estimate.
When we put together our estimates we obtain that $\vol(z_2)\leq p \binom{n}{k} (2^n-1) 2^k p^{\frac{k-n}{n}}$, from which we obtain the claimed bound 
\[
\vol(z)\leq \binom{n}{k} 2^{n+k} p^{k/n}.
\]

%Indeed, $\Psi_\epsilon(\pr_P^{-1}(x_i)\cap K)$ can be partitioned into a piece contained in $\skel_{n-k}(K)$ and its complement. Its enough to estimate the volume of the complement. 
%The set $\Psi_\epsilon$

%Partition $K$ according to the facets of $\phi^*$, the region of $\pr_P^{-1} \cap(x_i)\cap K$ corresponding to a fixed facet gets mapped to an algebraic surface of degree $2$, applying Crofton's formula we get  $\skel_{n-k}(K)$ containing $lim_{\epsilon \to 0}\Psi_\epsilon(x)$. Now the intersection of $\pr_P^{-1}(x_i)$ with $\Gamma^*+\epsilon$ that gets mapped to a fixed facet of $j$

 %$vol_{n-k}(\Psi_\epsilon(\pr_P^{-1}(x_i)\cap K)\sum_{\sigma \in \skel_{n-k}(K)}$

 %The $n-k$ flat $\pr_P^{-1}(x_i)$ is can be described by $k$ linear equations, Its preimage under $\Psi_\epsilon$ is contained in 

 %$\Psi_\epsilon [(\pr_P^{-1}(x_i)\cap K]$
%On the other hand $\vol_{n-k}(\skel_{n-k}(K)) \leq 2^{n+k/2}\binom{n}{k}^2$ estimating the number of $(n-k)$ faces to be less than $2^{n}\binom{n}{k}^2$ and the volume of each less than $2^{k/2}$ by a theorem of Keith Ball we can conclude that $\vol(z_2)\leq p (n-k) 2^{3n+k/2}  p^{-(n-k)/n}$. A small calculation yields the result.
\end{proof}
\begin{remark}
The same proof is the key step to show that for any set $U \subset \mathbb{R}^n$ the $(n-k)$-waist of $U$ is bounded above by $\binom{n}{k} 2^{n+k} \vol(U)^{\frac{n-k}{n}}$. By scaling, one can assume that $\vol(U)=1$. One then chooses a translation of $\Gamma$ such that $\vol_{n-k}(U \cap \skel_{n-k}(\Gamma+x))\leq \binom{n}{k}$. Crofton's formula implies that this is possible because the right hand side corresponds to the expected volume of the intersection when $x$ is chosen at random. Using the bending construction as just explained on $F_P$ with $p=1$ yields Guth's width-volume inequality.
\end{remark}

%The proof goes back to Gromov and Guth, to state it in more generality let $P_N(X)$ the set of measurable $N$-partitions of $X$. More precisely and element of $P_N(X)$ is a union of measurable subsets of $X$ that are pairwise disjoint and whose union has the same volume as $X$.

%\begin{theorem} Given a space $X$ and a cup product of cohomology classes $\alpha=\Pi \alpha_i$, with $\alpha_i \in H^*(Z_l(X))$,
%\[\vol(\alpha)\geq \sup_{p\in P_N(X)} \vol (\alpha_i|_p_i)$ 

\section{Coloring the cube and waists of maps to polyhedra}
\label{section:hausdorff-cube}

\subsection{Statement of the results about the cube}

\label{section:hauscube-statements}

Let us state theorems already mentioned in the introduction of the paper:

\begin{theorem}
\label{theorem:hauswaist}
There exists a constant $\varepsilon_{k,n}>0$ such that for every continuous map
$$
f : [0,1]^n\to Y
$$
from the unit cube $[0,1]^n$ to a $k$-dimensional polyhedron $Y$ there exists $y\in Y$ such that the set $f^{-1}(y)$ has $(n-k)$-Hausdorff measure $\haus^{n-k}(f^{-1}(y))$ at least $\varepsilon_{k,n}$.
\end{theorem}
If we allowed the constant $\varepsilon_{k,n}$ to depend on $Y$ then this sort of result would follow from the case $Y=\mathbb R^k$, just because any $k$-dimensional $Y$ can be mapped to $\mathbb R^k$ with bounded multiplicity.  A similar theorem was considered by Gromov in \cite[Appendix 2, B2']{grom1983}, see also \cite[Section 7]{guth2014}; we report our similar proof since some readers might profit from the connection to colorings of the cube.

%The idea of the proof of Theorem~\ref{theorem:hauswaist} comes from~\cite{kar2013cube} (see also an independent proof in~\cite{matdinov2013}) where, to answer a question of Matou\v{s}ek and P\v{r}\'{\i}v\u{e}tiv\'{y}~\cite{matpri2008}, Gromov's filling technique was used. Here we provide a continuous analogue of the main result  of ~\cite{kar2013cube}. %to show that there exists a constant $c_{n,k}$, such that whenever the discrete cube $[N]^n$ is colored with $k$ colors, there exists a monochromatic connected component of size at least $c_{n,k} N^{n-k}$.  Here we give provide a classical formulation, using closed coverings of the cube for which the proof in~\cite{kar2013cube} is applicable and it simplifies. 

In the case $k=1$, it was observed by Gromov that the Hausdorff measure estimate is established for $f : \mathbb S^n \to \mathbb R$ by taking $y$ to be the median value of $f$ and applying the spherical isoperimetric inequality for the Hausdorff measure. In the case $k=n-1$, Alexandrov's width (waist) theorem (see \cite[Section 6]{kar2012}) guarantees that a connected component of some fiber $f^{-1}(y)$ cannot be covered by a smaller homothet of the cube, hence its $1$-Hausdorff measure is at least $\varepsilon_{n-1,n} = 1$; this estimate is evidently tight. In general the result in~\cite[Section 1.3]{grom2010} establishes some constant $\varepsilon_{k,n}$ for maps into the Euclidean space, here we generalize it for maps to arbitrary $k$-dimensional polyhedron independently on the polyhedron. In general \emph{we have no evidence} that the constant $\varepsilon_{k,n}$ has to be less than $1$ in the case of general target spaces $Y$.\\
\subsection{Coverings of the cube}
Lebesgue's lemma claims that for any colouring of the $n$ dimensional cube with $n$ colors there exists a color intersecting opposite faces of the cube. A discrete version of this theorem goes by the name of \emph{the hex lemma}. Inspired by these results, Matou\v{s}ek and P\v{r}\'{\i}v\u{e}tiv\'{y}~\cite{matpri2008} raised the question of what happens when the number of colors $k$ is smaller than the dimension $n$. They conjectured a lower bound on the number of cubes on a monochromatic connected component of the discrete cube $\{0,1,2\ldots m\}^n$ with the topology of the standard lattice graph. Specifically, they conjectured that the number of cubes in such component is of the order of $m^{n-k}$ where $k$ is the number of colors. This conjecture was confirmed by ~\cite{kar2013cube}, and independently by ~\cite{matdinov2013}. The following theorem is a continuous analogue of that theorem. Its proof is a simplified version of the proof for the discrete case.

\begin{theorem}
\label{theorem:cover-haus}
Suppose $[0,1]^n$ is covered by a finite family of closed sets $\{C_i\}$ with degree of the intersection graph\footnote{The \emph{intersection graph} of a covering is the $1$-skeleton of its nerve.} bounded by $\delta$. Then there exist $\varepsilon_{n,\delta}>0$, $k\in\{0,\ldots, n\}$, and indices $i_0<i_1<\dots<i_k$ such that the $(n-k)$-Hausdorff measure of the intersection 
$$
C_{i_0}\cap \dots \cap C_{i_k}
$$
is at least $\varepsilon_{n,\delta}$.
\end{theorem}

We cannot derive this theorem from Theorem~\ref{theorem:hauswaist}, but the core of both proofs is the same. %They are based of Lemmas~\ref{lemma:filling} and \ref{lemma:assignment} that, in the case of the waist theorem, give factors of order $2^{n-k}$ and $2^k$ at $k$ stages of the proof, resulting in the total estimate $\varepsilon_{n,k} \sim 2^{-nk}$. 

\subsection{Proof of Theorem~\ref{theorem:cover-haus}}

We will prove Theorem~\ref{theorem:cover-haus} by the same method as in~\cite{kar2013cube}, in fact, the proof here will be even simpler. Define the $k$-Hausdorff measure using covering by open cubes with edge-lengths $d_1,\ldots, d_N$ and considering the sum
$$
\sum_{i=1}^N d_i^k.
$$ 
We minimize such a sum over all coverings by cubes of size at most $\eta$, and then go to the limit $\eta \to 0$. The sum we use differs from the sum in the standard definition of the Hausdorff measure by a constant depending on $n$ and $k$, which is not relevant to the problem in question.

First, for every set $I$ of indices $i_0<i_1<\dots<i_k$ with $k\le d$ we consider 
$$
C_I = C_{i_0}\cap \dots \cap C_{i_k}
$$
and cover it by open cubes with edge-lengths $\{d_{j, I}\}_{1\le j \le N_I}$. Assuming the contrary the conclusion of the theorem we have, for any $\varepsilon$, a cover such that
$$
\sum_{j} d_{j, I}^{n-k} < \varepsilon
$$
for every non-empty $I$; this inequality also implies that the covering is sufficiently fine. The sets we cover are compact and the cubes we cover with are open, hence we may assume every $C_I$ is covered by a finite number of cubes. 

Now we follow the proof in~\cite{kar2013cube}, and put the $C_i$ in a \emph{general position} in a certain sense. We choose an open $U_i\supset C_i$ for every $U_i$, so that the respective intersections $U_I$ are still covered by the collections of cubes we start with. Then we take a partition of unity $\{f_i\}_i$ subordinate to $U_i$ that can be viewed as a continuous map $f$ from $[0,1]^n$ to the nerve $N$ of the covering $\{C_i\}$. We take the barycentric subdivision $N'$ of the nerve and consider the stars $D_i$ in $N'$ of every vertex $v_i\in N$. Every set $f^{-1}(D_i)$ is contained in its respective $U_i$. Then we approximate $f$ by a PL-map $g$, transversal to the partition $\{D_i\}$, and substitute $C_i$ by the PL-sets $C'_i = g^{-1}(S_i)$. The new sets $C'_i$ are polyhedra still contained in their respective $U_i$, they give a partition of the cube, and every intersection $C'_I$ is still contained in its respective union of open cubes of edge-lengths $\{d_{j, I}\}_j$. The general positions assumption on $g$ also implies that every intersection $C'_I$ has dimension at most $n+1-|I|$ and, in particular, is empty for $|I|>n+1$. In what follows we denote the modified $C'_i$ by $C_i$ again.

Now we are going to fill PL-cycles of the pair $([0,1]^n, \partial [0,1]^n)$ and use the appropriate modification of~\cite[Lemma 2.3]{kar2013cube}. If a subset $X\subset [0,1]^n$ is covered by a collection of cubes of sizes $\{d_j\}$ with $\sum_j d_j^k < \varepsilon$ we say that $X$ is \emph{$k$-dimensionally $\varepsilon$-covered}; this is a part of the definition of $k$-dimensional Hausdorff measure.

\begin{lemma}
\label{lemma:filling}
For any cycle $z\in Z_k([0,1]^n,\partial [0,1]^n)$ with $0\le k < n$, $k$-dimensionally $\varepsilon$-covered, one can find its filling $H(z)\in C_{k+1}([0,1]^n,\partial [0,1]^n)$ such that $\partial H(z) = z\pmod {\partial [0,1]^n}$ and $H(z)$ is $(k+1)$-dimensionally $A_{k,n}\varepsilon$-covered for some $A_{k,n}>0$. Additionally, the $A_{k,n}\varepsilon$-cover of $H(z)$ will only depend on the $\varepsilon$-cover of $z$ and not on $z$ itself.
\end{lemma}

The proof of the lemma is a minor adaptation of the argument in~\cite{kar2013cube}, so we postpone it and continue with the theorem. We take $\varepsilon < 1$ and, since every point has $0$-dimensional Hausdorff measure $1$, $C_I$ must be empty for $|I|> n+1$. Let $m$ be the maximal $|I|$ corresponding to nonempty $C_I$. The general position of the map $g$ and the parts $C_i$ implies
$$
\partial C_I = \sum_{i\not\in I} C_{\{i\}\cup I};
$$
this formula is understood as an equality of chains modulo $2$, relative to the boundary of the cube. The idea behind this formula is that, under the general position of the $C_i$, the intersections $C_I$ is locally a PL-submanifold of codimension $|I|-1$, unless it touches its sub-intersection $C_{\{i\}\cup I}$, more explanations about this can be found in~\cite{kar2013cube}. Every $C_I$ with maximal $|I|=m$ is, by the boundary formula, an $(n+1-m)$-dimensional relative cycle and by Lemma~\ref{lemma:filling}, we fill it with an $(n+2-m)$-dimensional chain $F_I$. Since $C_I$ is $(n+1-m)$-dimensionally $\epsilon$-covered then $F_I$ is $(n+2-m)$-dimensionally $A_{n+1-m, n}\varepsilon$-covered by some collection of cubes. After that we proceed by descending induction on $|I|$ and continue filling
$$
F_I = H\left(C_I - \sum_{i\not\in I} F_{\{i\}\cup I}\right)\ \text{to have}\ \partial F_I = C_I - \sum_{i\not\in I} F_{\{i\}\cup I}.
$$
The expression we fill is a cycle because we calculate by induction
$$
\partial \left(C_I - \sum_{i\not\in I} F_{\{i\}\cup I}\right) = \sum_{i\not\ni I} C_{\{i\}\cup I} - \sum_{i\not\in I} C_{\{i\}\cup I} + \sum_{i, j\not\in I} F_{\{i\}\cup \{j\}\cup I} = 0,
$$
the last double sum cancels modulo $2$, because the summands go in pairs of equal $F_{\{i\}\cup \{j\}\cup I}$ and $F_{\{j\}\cup \{i\}\cup I}$.

All such sets $F_I$ are $(n+2-|I|)$-dimensionally $A_{|I|}\varepsilon$-covered for some big constant $A_{|I|}$. Here we \emph{use the fact that the degree of the intersection graph is bounded} and therefore the sum, the argument of $H$, has a bounded number of summands. In~\cite{kar2013cube} there was no assumption of the bounded degree of the intersection graph, but there was Lemma~3.1 to give the required total estimate; it is not clear if we can avoid the assumption of bounded total degree in Theorem~\ref{theorem:cover-haus} here.

Eventually, the cycles
$$
X_i = C_i - \sum_{j\neq i} F_{i,j}
$$
will be $n$-dimensionally $A\varepsilon$-covered for some accumulated constant $A$. If we chose $\varepsilon < 1/A$ then every such cycle $X_i$ has volume less than $1$ and is therefore zero in $H_n([0,1]^n,\partial [0,1]^n)$. So their sum must also be zero in the homology, but from the evident equality $F_{i,j} = F_{j,i}$, working modulo $2$ we obtain
$$
\sum X_i = \sum_i C_i - \sum_{i\neq j} F_{i,j} = \sum_i C_i,
$$
which is the nontrivial generator of $H_n([0,1]^n,\partial [0,1]^n)$. This contradiction proves Theorem~\ref{theorem:cover-haus}.

\subsection{Proof of Theorem~\ref{theorem:hauswaist}}

Assume we have a small $\varepsilon>0$ such that every $f^{-1}(y)$ is $(n-k)$-dimensionally $\varepsilon$-covered by a set of open cubes (the terminology is the same as in the above section). Then $f^{-1}(y)$ is covered along with its neighborhood and from compactness it follows that there exists a neighborhood $U_y\ni y$ such that $f^{-1}(U_y)$ is $(n-k)$-dimensionally $\varepsilon$-covered.

Such open sets $U_y$ cover $Y$ and we are going to find a partition $Y=D_1\cup\dots\cup D_N$ of $Y$ into compact sets such that no point of $Y$ is covered more than $k+1$ times and every $D_j$ is contained in certain $U_y$. In order to have every $D_j$ contained in $U_y$ we just take a triangulation $T$ of $Y$ and make barycentric subdivisions of it many times setting $T_m = T_{m-1}'$. Put $D_j$ to be the stars of vertices of $T_m$ in $T_{m+1} = T_m'$. 
We claim that for sufficiently large $m$ each star $D_j$ must be contained in a single open set $U_y$. Otherwise we subdivide infinitely many times, since the diameter of the stars is strictly decreasing, compactness implies that some sequence of sets $\overline D_m$ must converge to a point $y_0$, but then $U_{y_0}$ contained already $\overline D_m$ for sufficiently large $m$.

Put $C_j = f^{-1}(D_j)$, from the above construction every such set and every intersection of several of them is $m$-dimensionally $\varepsilon$-covered, for every $m\ge n-k$. Now the proof of Theorem~\ref{theorem:cover-haus} would work if we guaranteed that the incidence graph of the covering $\{D_j\}$ (and therefore $\{C_j\}$) has degree bounded (in terms of $k$). In general the bounded degree cannot be achieved, but we modify the argument to remedy this. In the formula
\begin{equation}
\label{equation:filling}
F_I = H\left(C_I - \sum_{i\not\in I} F_{\{i\}\cup I}\right).
\end{equation}
the number of summands is not bounded, but we observe that the proof of Lemma~\ref{lemma:filling} works independently of the cycle $z$, the filling can be made to depend only on the collection of cubes covering $z$. We need to guarantee that the right hand side of (\ref{equation:filling}) only invokes a bounded number of such collections of cubes. The collections of cubes we use are in the correspondence with the sets $C_i$. We need a scheme of assigning such collections to every index sets $I\cup\{i\}$ so that only a bounded number of the collections is invoked in the formula (\ref{equation:filling}) for every $I$. Such a scheme is guaranteed by the following:

\begin{lemma}
\label{lemma:assignment}
Let $T$ be a triangulation of a $k$-dimensional polyhedron and let $T'$ be its barycentric subdivision. For every face $\sigma$ of $T'$, it is possible to assign a vertex, $v=f(\sigma)$, of $T'$ so that $v\in \sigma$ and all the sets of vertices
$$
N_v = \{f(\sigma) : \sigma\ni v\}
$$
have sizes bounded in terms of $k$.
\end{lemma}

\begin{proof}
Note that a vertex $v$ of $T'$ is a face of $T$ by definition, and a face of $T'$ is a collection of faces of $T$ linearly ordered by inclusion, that is 
$$
\sigma = \{v_0<\dots < v_m\}.
$$
Let the function $f$ just assign to every $\sigma$ its minimal element, $v_0$ in the above expression. Now the assumption $\sigma\ni v$ and $w = f(\sigma)$ implies $w<v$. In terms of $T$, this means that the face $w$ is a subface of $v$. But for a given $v$ the number of such subfaces $w$ is at most $2^k-1$, that is bounded in terms of $k$.
\end{proof}

\subsection{Proof of Lemma~\ref{lemma:filling}}

Here we follow the proof of~\cite[Lemma 2.3]{kar2013cube} almost literally.
\begin{proof}
 Let $z$ be covered by the cubes $Q_1,\ldots, Q_N$ of sizes $d_1,\ldots,d_N$. 
Consider all possible intersections $z_t=z\cap \{x_1=t\}$ with a hyperplane orthogonal to the $0x_1$ axis. Put 
$$
J(t) = \{i : Q_i\cap \{x_1=t\} \neq \emptyset\},\quad S(t) = \sum_{j\in J(t)} d_j^{k-1}.
$$
It is easy to see that 
$$
\int_0^1 S(t) = \sum_j d_j^k < \varepsilon
$$
and therefore, for some $t$, $S(t)<\varepsilon$. Hence the set $z_t$ is $(k-1)$-dimensionally $\varepsilon$-covered by $\{Q_j\}_{j\in J(t)}$, for generic $t$ this $z_t$ must be a $(k-1)$-cycle in $([0,1]^n,\partial [0,1]^n)$.

Now we cut the cycle $z$ into two chains $z = z_{\le t} + z_{\ge t}$, where $z_{\le t}$ corresponds to $x_1 \le t$ and $z_{\ge t}$ corresponds to $x_1 \ge t$. The cutting is applicable because after cutting every simplex of $z$ we may triangulate the cut parts to make it again a simplicial chain. It is clear that
$$
\partial z_{\le t} = z_t \pmod{\partial [0,1]^n}\quad\text{and}\quad  \partial z_{\ge t} = -z_t \pmod{\partial [0,1]^n}.
$$

Let $F_0$ and $F_1$ be the facets of $[0,1]^n$ with $x_1=0$ and $x_1=1$ respectively. If a chain $y\in C_\ell([0,1]^n,\partial [0,1]^n)$ is given as a PL image of a simplicial complex $y : K\to [0,1]^n$ then we construct a PL map $I_0(y) : K\times [0, 1] \to [0,1]^n$ by sending $p\times 1$ to $y(p)$, $p\times 0$ to the projection of $y(p)$ onto $F_0$, and $p\times s$ to the corresponding combination of $I_0(y)(p\times 0)$ and $I_0(y)(p\times 1)$. If $y$ does not touch $F_1$ then we have:
$$
\partial I_0(y) = y + I_0(\partial y)\pmod{\partial [0,1]^n}.
$$
Similarly we define $I_1(y)$ with the projection onto $F_1$ for those $y$ that do not touch $F_0$ with
$$
\partial I_1(y) = y - I_1(\partial y)\pmod{\partial [0,1]^n}.
$$
Now return to our cycle $z=z_{\le t}+z_{\ge t}$ and put
$$
H(z) = I_0(z_{\le t}) + I_1(z_{\ge t}) - H(z_t)\times [0, 1].
$$
For the boundary we have (modulo $\partial [0,1]^n$):
$$
\partial H(z) = \partial I_0(z_{\le t}) + \partial I_1(z_{\ge t}) - z_t\times [0,1] = z_{\le t} + I_0(z_t) + z_{\ge t} - I_1(z_t) - z_t\times [0,1] = z,
$$
because $I_0(z_t) - I_1(z_t)$ obviously equals $z_t\times [0,1]$. It remains to make an economic covering of $H(z)$ by cubes.

The parts $I_0(z_{\le t})$ and $I_1(z_{\ge t})$ will we covered if we extend the cubes in the original collection $\{Q_j\}$ so that $x_1$-coordinate spans $[0,1]$, making every cube $Q_j$ a long box, that is coverable by at most
$$
\left\lceil\frac{1}{d_j}\right\rceil \le \frac{2}{d_j}
$$ 
cubes of size $d_j$. If those new cubes are denoted by $Q'_{j,\ell}$ with $\ell$ from $1$ to something at most $\frac{2}{d_j}$ then for their sizes we have
$$
\sum_{Q'_{j,\ell}} d_j^{k+1} \le \sum_j \frac{2}{d_j} d_j^{k+1} = 2 \sum_j d_j^k,
$$ 
so $I_0(z_{\le t}) + I_1(z_{\ge t})$ is $(k+1)$-dimensionally $2\varepsilon$-covered.

The cycle $z_t$ is considered to be in the cube $Q' = Q\cap \{x_1=t\}$ of lower dimension; it has $(k-1)$-dimensional and $\varepsilon$-covering by the choice of $t$. $H(z_t)$ is defined by the inductive assumption and is $k$-dimensionally $A_{k-1,n-1}\varepsilon$-covered by some cubes. The multiplication by the segment $[0,1]$ produces $(k+1)$-dimensional chains in $Q$ from chains in $Q'$ in an obvious way. Again, every covering cube of size $d_j$ after this has to be replaced by approximately 
$$
\left\lceil\frac{1}{d_j}\right\rceil \le \frac{2}{d_j}
$$ 
cubes of the same size, so the cylinder $H(z_t)\times [0,1]$ gets $(k+1)$-dimensionally $2A_{k-1,n-1}\varepsilon$-covered.

In conclusion, if we put 
$$
A_{k,n} = 2A_{k-1,n-1} + 2.
$$

then $H(z)$ is $(k+1)$-dimensionally $A_{k,n}$-coverable \end{proof}

\bibliography{../Bib/karasev}
\bibliographystyle{abbrv}
\end{document}